\documentclass[11pt,reqno]{amsart}
\usepackage{ytableau,tikz,hyperref,amsthm}
\usepackage{mathtools}
\usepackage{palatino}
\usepackage[left=3cm,right=3cm,top=3cm,bottom=3cm]{geometry}
\usepackage[shortlabels,inline]{enumitem}
\newcommand\characx{\mathbf{x}}
\newcommand{\key}{\kappa}
\newtheorem{theorem}{Theorem}
\newtheorem{remark}{Remark}
\newcommand\bremark{\begin{remark}\begin{upshape}}
\newcommand\eremark{\end{upshape}\end{remark}}
\newtheorem{proposition}{Proposition}
\newtheorem{corollary}{Corollary}

\newtheorem{definition}{Definition}
\newtheorem{example}{Example}
\DeclareMathOperator{\CT}{CT}
\DeclareMathOperator{\row}{row}
\DeclareMathOperator{\Tab}{Tab}
\DeclareMathOperator{\Mat}{Mat}
\DeclareMathOperator{\SVT}{SVT}
\DeclareMathOperator{\RPP}{RPP}
\DeclareMathOperator{\SVRPP}{SVRPP}
\DeclareMathOperator{\ceq}{ceq}
\DeclareMathOperator{\wt}{wt}
\DeclareMathOperator{\h}{ht}
\DeclareMathOperator{\ex}{ex}
\DeclareMathOperator{\Read}{rw}
\DeclareMathOperator{\cw}{cw}
\DeclareMathOperator{\rect}{rect}
\DeclareMathOperator{\sh}{shape}
\allowdisplaybreaks
\title[]{Demazure crystals for flagged set-valued reverse plane partitions}
\author[]{Siddheswar Kundu}
\address{Department of Mathematics, Indian Institute of Technology Kanpur, Kanpur 208016, India.}
\email{kundusidhu96@gmail.com}
\begin{document}
\begin{abstract}
In this paper, we build a Demazure crystal structure on the set of all flagged set-valued reverse plane partitions of a given skew shape $\lambda/\mu$ and flag $\Phi$. Our construction provides a unified generalization of the Demazure crystal structures previously known for the set of all flagged semi-standard set-valued tableaux and the set of all flagged reverse plane partitions. Consequently, we obtain a combinatorial expansion for the flagged hybrid Grothendieck polynomial $H_{\lambda/\mu}(\mathbf{x}_{\Phi};\mathbf{t};\mathbf{w})$, introduced by Guo--Kang--Liu, in terms of key polynomials. Applying this expansion, we express the hybrid Grothendieck polynomial $H_{\lambda/\mu}(\mathbf{x};\mathbf{t};\mathbf{w})$ in terms of both the stable Grothendieck polynomials $G_{\nu}(\mathbf{x})$ and the dual stable Grothendieck polynomials $g_{\nu}(\mathbf{x})$. 
\end{abstract}
\maketitle
\section{Introduction}
Lascoux and Schützenberger \cite{Lascoux:G1,Lascoux:G2} introduced Grothendieck polynomials, indexed by permutations in the symmetric group $S_n $, to describe the K-theory ring of the Grassmannian. These polynomials serve as a K-theory analogue of Schubert polynomials. By taking the stable limit of $n \rightarrow \infty $, Grothendieck polynomials yield symmetric functions. In the K-theory of the Grassmannian, the classes of structure sheaves of Schubert varieties are represented by the stable Grothendieck polynomials. We denote the stable Grothendieck polynomial corresponding to the Grassmannian permutation $\pi_{\lambda}$ (\cite[\S2]{Buch-KLR}) by $G_{\lambda}(\mathbf{x})$. Buch \cite{Buch-KLR} provided an explicit combinatorial interpretation of $G_{\lambda}(\mathbf{x})$ as the generating function for semi-standard set-valued tableaux of shape $\lambda$. The dual stable Grothendieck polynomials $g_{\lambda}(\mathbf{x})$ are Hopf-dual to $G_{\lambda}(\mathbf{x})$ with respect to the Hall inner product, and they realize the classes in the K-homology of ideal sheaves of the boundaries of Schubert varieties. Lam and Pylyavskyy have given an explicit combinatorial formula for $g_{\lambda}(\mathbf{x})$ using reverse plane partitions in \cite{Lam}. Chan--Pflueger \cite{Melody-Nathan} introduced the refined skew stable Grothendieck polynomial $G_{\lambda/\mu}(\characx;\mathbf{w}) $ and Galashin--Grinberg--Liu \cite{Galashin:Bender-Knuth} studied the refined dual stable Grothendieck polynomial $\Tilde{g}_{\lambda/\mu}(\characx;\mathbf{t})$. Guo--Kang--Liu introduced the hybrid Grothendieck polynomial $H_{\lambda/\mu}(\characx;\mathbf{t};\mathbf{w})$ in \cite{Hybrid}, which is a common extension of the refined skew stable Grothendieck polynomial $G_{\lambda/\mu}(\characx;\mathbf{w}) $ and the refined dual stable Grothendieck polynomial $\Tilde{g}_{\lambda/\mu}(\characx;\mathbf{t})$. See Section~\ref{Section 2} for all the notations.  

The flagged hybrid Grothendieck polynomial $ H_{\lambda/\mu}(\characx_{\Phi};\mathbf{t};\mathbf{w})$ \cite{Hybrid} is a common generalization of the flagged refined skew stable Grothendieck polynomial $G_{\lambda/\mu}(\characx_{\Phi};\mathbf{w}) $\cite{Sidhu:SVT} and the row-flagged refined dual stable Grothendieck polynomial $\Tilde{g}^{\row(\mathbf{1},\Phi)}_{\lambda/\mu}(\characx;\mathbf{t})$ \cite{Galashin:Bender-Knuth}. Reiner and Shimozono \cite[Theorem~20]{RS} gave an expansion of the flagged skew Schur polynomials $s_{\lambda}(X_{\Phi})$ in the basis of key polynomials. This result was later extended to both $G_{\lambda/\mu}(\characx_{\Phi};\mathbf{w})$ (see \cite[Corollary~2]{Sidhu:SVT}) and $\Tilde{g}^{\row(\mathbf{1},\Phi)}_{\lambda/\mu}(\characx;\mathbf{t})$ (\cite[Corollary~4]{Sidhu:SVT}). The proofs were established by demonstrating that a Demazure crystal structure can be endowed upon the respective indexing sets: the flagged semi-standard set-valued tableaux \cite[Remark~3]{Sidhu:SVT} and the flagged reverse plane partitions \cite[Theorem~1]{Sidhu}.

In this article, we answer two open problems raised by Guo, Kang, and Liu. First, we establish a Demazure crystal structure on the set of flagged set-valued reverse plane partitions, resolving \cite[Problem 6.5]{Hybrid}. As a consequence, we derive the key polynomial expansion of the flagged hybrid Grothendieck polynomial $ H_{\lambda/\mu}(\characx_{\Phi};\mathbf{t};\mathbf{w})$. Applying this result, we express the hybrid Grothendieck polynomial $H_{\lambda/\mu}(\characx;\mathbf{t};\mathbf{w})$ in the basis of both $G_{\nu}$'s and $g_{\nu}$'s, thereby settling \cite[Problem 6.4]{Hybrid}.

This paper is organized as follows. In \S\ref{Section 2}, we review the background on crystals, set-valued reverse plane partitions, and hybrid Grothendieck polynomials, along with the existing crystal structure on set-valued reverse plane partitions. In \S\ref{Section 3}, we establish the main theorem (Theorem~\ref{theorem:main}) and deduce the key polynomial expansion of $H_{\lambda/\mu}(\characx_{\Phi};\mathbf{t};\mathbf{w})$. Finally, in \S\ref{Section 4}, we provide the expansions of $H_{\lambda/\mu}(\characx_{\Phi};\mathbf{t};\mathbf{w})$ in the basis of both $G_{\nu}$'s and $g_{\nu}$'s.
\section{Preliminaries}
\label{Section 2}
We denote the sets of integers, non-negative integers, and positive integers by $\mathbb{Z}, \mathbb{Z}_+, \mathbb{N}$ respectively. Similarly, $\mathbb{Z}^n, \mathbb{Z}_+^n $ represent the sets of $n$-tuple of integers and non-negative integers respectively. Also, for $n \in \mathbb{N}$, we set $[n]:=\{1,2,\dots,n\}$.

A \emph{crystal} of type $A_{n-1}$ consists of a non-empty finite set $\mathcal{B}$ along with the maps $ \wt: \mathcal{B} \rightarrow \mathbb{Z}^{n}$ and
$ e_i,f_i :\mathcal{B} \rightarrow \mathcal{B} \sqcup \{0\}$ for $i \in [n-1] $, where $0 \notin \mathcal{B}$ such that:
\begin{enumerate}
    \item $e_i(x)=y \text{ if and only if } f_i(y)=x$ for all $x,y \in \mathcal{B}$ and $i \in [n-1]$, in which case, $\wt(y)-\wt(x)= \epsilon_i-\epsilon_{i+1},$ where $ \epsilon_i \in \mathbb{Z}^{n}$ denotes the $i$-th standard basis vector.
    \item For $x \in \mathcal{B}, j \in [n-1]$, we define the following maps $\varepsilon_j,\varphi_j: \mathcal{B} \rightarrow \mathbb{Z}_+$   
    $$\varepsilon_j(x)=\max\{k|e_j^{k}x\neq 0\}; \quad 
 \varphi_j(x)=\max\{k|f_j^{k}x\neq 0\}.$$
 Then $\varphi_j(x) - \varepsilon _j (x) = \wt(x) \cdot (\epsilon_j - \epsilon_{j+1})$ $\forall j \in [n-1]$, $ x \in \mathcal{B}$.
\end{enumerate}
We often identify a crystal with its underlying set $\mathcal{B}$. The notion of a crystal presented here corresponds to a seminormal crystal defined by Bump and Schilling \cite[\S2.2]{Bump-Anne}. To any type $A_{n-1}$ crystal $\mathcal{B}$, we associate a directed edge-colored graph called the \emph{crystal graph} of $\mathcal{B}$, whose vertex set is $\mathcal{B}$ and directed edges are given by $x \xrightarrow{\,i\,} y$ if and only if $f_i(x) = y$. 
\begin{itemize}
    \item We call $\mathcal{B}$ is \emph{connected} if its underlying undirected crystal graph is connected.
    \item A subset $\mathcal{B}' \subseteq \mathcal{B}$ formed by a union of connected components of $\mathcal{B}$ inherits a crystal structure from $\mathcal{B}$ and is called a \emph{full subcrystal} of $\mathcal{B}$.
\end{itemize}
\begin{example}
    \label{ex:crystal}
Let $\mathcal{W}_n^m$ be the set of all words of length $m$ with letters in $[n]$. For a word $u=u_1\cdots u_m \in  \mathcal{W}_n^m$, the weight $\wt(u)$ is the $n$-tuple $(q_1,\dots,q_n)$, where $q_i$ counts the number of occurrences of $i$ in $u$.

To define the crystal operators $e_i,f_i : \mathcal{W}_n^m \rightarrow \mathcal{W}_n^m \sqcup \{ 0\}$ for $i \in [n-1] $, we apply the signature rule: consider the subword of $u$ consisting only of the letters $i$ and $i+1$, and then replace each $i$ by $\textcolor{blue}{+}$ and each $i+1$ by $\textcolor{red}{-}$, and iteratively cancel all adjacent $(\textcolor{red}{-}, \textcolor{blue}{+})$ pairs. The uncancelled entries yield a string of the form $\textcolor{blue}{+}\cdots\textcolor{blue}{+}\textcolor{red}{-}\cdots\textcolor{red}{-} $. We then define $e_i,f_i$ as follows:
\begin{itemize}
    \item $e_i(u):$ If the final string contains no uncancelled $\textcolor{red}{-}$, then $e_i(u)=0$. Otherwise, $e_i(u)$ is the word obtained from $u$ by changing the letter $i+1$ corresponding to the leftmost uncancelled $\textcolor{red}{-}$ into $i$.
    \item $f_i(u):$ If the final string contains no uncancelled $\textcolor{blue}{+}$, then $f_i(u)=0$. Otherwise, $f_i(u)$ is the word obtained from $u$ by changing the letter $i$ corresponding to the rightmost uncancelled $\textcolor{blue}{+}$ into $i+1$.
\end{itemize}
Then $\Big( \mathcal{W}_n^m, \wt, \{e_i\}_{i \in [n-1]}
, \{f_i\}_{i \in [n-1]} \Big)$ is a type $A_{n-1}$ crystal.  
\end{example}
An element $b \in \mathcal{B}$ is called a \textit{highest weight} element if $e_i(b)=0$ for $i \in [n-1]$. For example, the word $1321$ a highest weight element of the crystal $\mathcal{W}_3^4$.

Let $\mathcal{A},\mathcal{B}$ be two type $A_{n-1}$ crystals. Then a \emph{crystal morphism} $\psi:\mathcal{A} \rightarrow \mathcal{B} \sqcup \{0\}$ is a map  such that $\forall  a \in \mathcal{A}$ with $\psi(a)\neq 0$ and $\forall i \in [n-1]:$
\begin{enumerate}
\item $\wt(\psi(a))=\wt(a)$,
$\varphi_i(\psi(a))=\varphi_i(a)$ and 
$\varepsilon_i(\psi(a))=\varepsilon_i(a)$;
\item $\psi(e_i a)=e_i\psi(a)$ whenever provided $e_ia \neq 0$ and $\psi(e_ia) \neq 0$;
\item $\psi(f_i a)=f_i\psi(a)$ whenever provided $f_ia \neq 0$ and $\psi(f_ia) \neq 0$.
\end{enumerate}
We say a morphism $\psi$ is \emph{strict} if it commutes with $e_i,f_i$ for $i \in [n-1]$. Furthermore, $\psi$ is called an \emph{embedding} or \emph{isomorphism} if the induced map $\psi : \mathcal{A} \sqcup \{0\} \rightarrow \mathcal{B} \sqcup \{0\}$ with $\psi(0)=0$ is a injective or bijective respectively.

A \emph{partition} $\lambda=(\lambda_1,\lambda_2,\ldots)$ is a weakly decreasing sequence of non-negative integers with finite sum $|\lambda|=\sum_i \lambda_i$. We identify $\lambda$ with its \emph{Young diagram}, which is a collection of top and left justified  boxes such that the $i^{th}$ row has length $\lambda_i$. For two partitions $\lambda, \mu,$ we write $ \mu \subseteq \lambda$ if $\lambda _i \geq  \mu_i $ $\forall i \geq 1$. For $\mu \subseteq \lambda $, the \emph{skew shape} $\lambda / \mu$ is obtained by deleting the boxes of $\mu$ from those of $\lambda$.
\begin{definition}
\begin{upshape}
    \cite[\S1]{Hybrid}
\end{upshape}
A \emph{set-valued reverse plane partition} (SVRPP) of shape $\lambda / \mu$ is a filling $H$ to each box $(i,j) \in \lambda/\mu$ by a non-empty finite subset $H_{i,j}$ of $\mathbb{N}$ such that
\begin{itemize}
    \item $ \max H_{i,j} \leq \min H_{i,j+1}$ whenever $ (i,j),(i,j+1) \in \lambda/\mu$,
    \item $ \max H_{i,j} \leq \min H_{i+1,j}$ whenever $ (i,j),(i+1,j) \in \lambda/\mu$.
\end{itemize}
\end{definition}
\bremark
It is worth noting that a SVRPP specializes to
\begin{enumerate}
    \item a semi-standard set-valued tableau \cite{Buch-KLR} when the column strictness condition holds.
    \item a reverse plane partition \cite{Lam} when every set contains exactly one element.
    \item a semi-standard Young tableau when the column strictness condition holds along with every box contains exactly one element.
\end{enumerate}
\eremark
Let $H$ be a SVRPP of shape $\lambda/\mu$. Then we define the following:
\begin{itemize}
    \item The \emph{weight} of $H$, denoted by $\wt(H) $, defined as the sequence of non-negative integers:
    $$\wt(H)=(l_1,l_2,\dots),$$ where each entry $l_i$ counts the total number of columns in $H$ that contain atleast one set containing the entry $i$.
    \item A box $(i,j) \in \lambda/\mu$ is said to be redundant subject to the following conditions:
    $$ (i+1,j) \in \lambda/\mu, \quad \max H_{i,j} = \min H_{i+1,j}.$$
 Then the \emph{column equalities vector} of $H$ is $\ceq(H) = (c_1,c_2,\dots)$ such that $ c_i$ is the number of redundant boxes in $i^{th}$ row of $H$.
    \item The \emph{excess vector} of $H$ is denoted by $\ex(H)\in \mathbb{Z}_+^n$, whose $i^{th}$ entry is defined by 
    $$\displaystyle\sum_{j:(i,j)\lambda/\mu}\big(|H_{i,j}|-1\big),$$ for each $i \geq 1$.
\end{itemize}
\begin{example}
    \label{ex:SVRPP}
  Consider the set-valued reverse plane partition of shape $(4,3,2,2)/(1)$ below
$$
\begin{tikzpicture}[scale=1.25]
    \draw (0,0)--(0,2.4)--(0.8,2.4)--(0.8,3.2)--(3.2,3.2)--(3.2,2.4)--(2.4,2.4)--(2.4,1.6)--(1.6,1.6)--(1.6,0)--(0,0);
    \draw (0,0.8)--(1.6,0.8);
    \draw (0,1.6)--(1.6,1.6);
    \draw (0.8,2.4)--(2.4,2.4);
    \draw (0.8,0)--(0.8,2.4);
    \draw (1.6,1.6)--(1.6,3.2);
    \draw (2.4,2.4)--(2.4,3.2);
    \node at (0.4,0.4) {$\textcolor{red}{2},3,5 $};
    \node at (1.2,0.4) {$6,7$};
    \node at (0.4,1.2) {$\textcolor{red}{1},2 $};
    \node at (1.2,1.2) {$\textcolor{red}{3},4,5 $};
    \node at (0.4,2) {$1 $};
    \node at (1.2,2-0.025) {$\textcolor{red}{2},3 $};
    \node at (2,2) {$4 $};
    \node at (1.2,2.8) {$1,2 $};
    \node at (2,2.8) {$2,3 $};
    \node at (2.8,2.8+0.025) {$3 $};
\end{tikzpicture}.
$$
Then $\wt(H)=(2,3,4,2,2,1,1)$, $\ceq(H)=(1,2,1,0)$ and $\ex(H)=(2,1,3,3)$.
\end{example}
\begin{definition}
\begin{upshape}
    \cite[\S3.3 \& \S3.4]{Hybrid}
\end{upshape}
The \emph{reading word} $\Read(H)$ of a SVRPP $H$ is the word formed by scanning its rows, from bottom to top, according to the following rule within each row: first, discard $\min(B)$ from each box $B$ in the row and scan the remaining elements from right to left, ordered decreasingly within each box. Then, scan the values $\min(B)$ across the row from left to right, omitting $\min(B)$ whenever it coincides with $\max(B')$ for the box $B'$ immediately above $B$.

The \emph{height vector} $\h(H)$ of $H$ is the string of row positions corresponding to the entries in $\Read(H)$. Explicitly, if $\Read(H)=v_1v_2\cdots$, then $\h(H)=h_1h_2\cdots$ such that $h_i$ is row index of $v_i$ for each $i \geq 1$.
\end{definition}
\begin{example}
   \label{ex:read}
The reading word of $H$ in Example~\ref{ex:SVRPP} is $\Read(H)= 753.6|542|3.14|32.123$. The ignored minimum entries of $H$ are shown in red. We have used vertical lines to separate subwords corresponding to different rows, and dots to separate non-minimal elements from minimal elements. It is clear that $\h(H)=444433322211111$.
\end{example}
\bremark
The reading word proposed here simultaneously extends the reading words of a semi-standard set-valued tableau, and of a reverse plane partition defined in \cite{Morse-Jason:K-bases}. We also note that, for reverse plane partitions, it differs from the reading word defined in \cite{Galashin:LR}, where the entries equal to the entry immediately below it are ignored.
\eremark
\begin{proposition}
\begin{upshape}
    \cite[Proposition 3.12]{Hybrid}
\end{upshape}
    \label{Prop:unique}
    Fix a skew shape $\lambda/\mu$ and sequences $u, h, \beta$. Then there exists at most one SVRPP $H$ of shape $\lambda/\mu$ such that $\Read(H)=u$, $\h(H)=h$ and $\ex(H)=\beta$. 
\end{proposition}
\subsection{Crystal structure on set-valued reverse plane partitions}
\label{def:crystal-SVRPP}
Fix $\alpha, \beta \in \mathbb{Z}^{n}_+$ and let $\lambda/\mu$ be a skew shape such that $\lambda,\mu \in \mathcal{P}_n$. We denote $|\alpha|$ by $\sum_{i }\alpha_i$ and similarly $|\beta|$ by $\sum_{i}\beta_i$.

Let $\SVRPP^{\alpha}_{\beta}(\lambda/\mu,n)$ be the set of all set-valued reverse plane partitions with entries at most $n$, together with $\ceq(H)=\alpha, \ex(H)=\beta$. We note that $\SVRPP^{\alpha}_{\beta}(\lambda/\mu,n)$ could be empty. For instance, $\SVRPP^{(2,0)}_{(0,1)} \big( (2,2)/(1,0), 3\big) $ is empty set. Then we can produce a structure of type $A_{n-1}$ crystal on 
$$ \SVRPP_n(\lambda/\mu): = \bigsqcup_{\alpha,\beta \in \mathbb{Z}_+^n} \SVRPP^{\alpha}_{\beta}(\lambda/\mu,n) ,$$ by embedding $\SVRPP^{\alpha}_{\beta}(\lambda/\mu,n)$ into $\mathcal{W}_n^{|\lambda|-|\mu|-|\alpha| +|\beta|}$, via the following map
$$ H \mapsto \Read(H).$$
The image of $\SVRPP^{\alpha}_{\beta}(\lambda/\mu,n) $ under this embedding is a full subcrystal of $\mathcal{W}_n^{|\lambda|-|\mu|-|\alpha| +|\beta|}$ (see \cite[\S3.2]{Hybrid}). The explicit crystal structure on $ \SVRPP_3\big((2,2)/(1,0)\big)$ can be found in \cite[Appendix~A]{Hybrid}.
\bremark
The crystal structure on $\SVRPP_n(\lambda/\mu)$ defined above, simultaneously generalizes the crystal structure on semi-standard set-valued tableaux introduced by Monical, Pechenik, and Scrimshaw \cite{Travis:symmetric} and that on reverse plane partitions due to Galashin \cite{Galashin:LR}.
\eremark
\bremark
\cite[Remark 3.11 \& \S3.5]{Hybrid}
\label{remark:stat}
The crystal operators $e_i, f_i$ on $\SVRPP_n(\lambda/\mu)$ preserve the height vector, $\ceq$ statistic and $\ex$ statistic.
\eremark
\bremark
\label{remark:crystal}
By definition of the crystal structure on set-valued tableau in \S\ref{def:crystal-SVRPP}, for each $i$, we have $ \Read(e_i. H)=e_i.\Read(H)$ and $ \Read(f_i. H)=f_i.\Read(H)$ for any $H \in \SVRPP_n(\lambda/\mu)$. 
\eremark
A word $w = w_1 w_2 \cdots w_s$ is said to be \textit{Yamanouchi}~\cite[\S 5.2]{Fulton:yt} if, for every $t \geq 1$ and $i \geq 1$, the suffix $w_t \cdots w_s$ contains at least as many occurrences of $i$ as of $(i+1)$. For instance, $4231211$ is a Yamanouchi word, whereas $1312$ is not.
\begin{proposition}
\begin{upshape}
    \cite[Proposition 3.8]{Hybrid}
\end{upshape}
    \label{Prop:Yama}
For $H \in \SVRPP_n(\lambda/\mu)$, $H$ is a highest weight element if and only if $\Read(H)$ is Yamanouchi word.
\end{proposition}
Let $\mathcal{P}_n$ be the set of all partitions with at most $n$ parts. Also, let $\Tab_n(\lambda) $ be the set of all semi-standard Young tableaux of shape $\lambda$ with entries at most $n$.
\begin{definition}
    \label{def: Demazure}
For $w \in S_n, \lambda \in \mathcal{P}_n,$ the \emph{Demazure crystal} $\mathcal{B}_{w} (\lambda)$ is defined as:
\begin{equation*}
\label{eq:demcrys}
  \mathcal{B}_{w} (\lambda) := \{f_{i_1} ^{m_1} f_{i_2} ^{m_2} \cdots f_{i_p} ^{m_p}  T_{\lambda} : m_j \geq0 \text{ for }j \in [p] \} \setminus \{0\} \subseteq  \Tab_n(\lambda),  
\end{equation*}
where $s_{i_1} s_{i_2} \cdots s_{i_p}$ is any reduced expression for $w$ and  $T_{\lambda}$ is the unique semi-standard Young tableau of shape $\lambda$ along with $\wt(T_{\lambda})=\lambda$.     
\end{definition}
\bremark
\cite[Theorem 13.5]{Bump-Anne}
$\mathcal{B}_{w} (\lambda)$ is independent of the choice of reduced expression of $w$.
\eremark
For $i \in [n-1]$, the $ i^{th}$ Demazure operator $D_i$ acting on the polynomial ring $\mathbb{Z}[x_1,x_2,\dots,x_n]$ is given by: 
$$ D_i(f):= \frac{x_i\, f - \, x_{i+1}s_i(f)}{x_i - x_{i+1}},$$
where $s_i$ acts on $f$ by interchanging $x_i$ and $x_{i+1}$. For a permutation $w \in S_n$, we define $D_w :=  D_{i_1}D_{i_2} \cdots D_{i_k},$ where $s_{i_1}s_{i_2} \cdots s_{i_k}$ is a reduced expression for $w$. As the operators $D_i$ satisfy the braid relations, $D_w$ is independent of the choice of reduced expression of $w$.
\begin{definition}
  For any $\alpha \in \mathbb{Z}^n _{+} ,$ the \emph{key polynomial} is defined by $\key_{\alpha}:=D_w(\characx^{\alpha^{\dagger}}),$ where $\alpha^{\dagger}$ is the unique partition in $S_n$-orbit of $\alpha$ and $w \in S_n$ is any permutation such that $w.\alpha^{\dagger}=\alpha$. Here, $S_n$ acts on $n$-tuples via the standard left permutation action $w.(\beta_1,\dots,\beta_n)=(\beta_{w^{-1}(1)},\dots, \beta_{w^{-1}(n)})$.
\end{definition}
\begin{proposition} 
\begin{upshape}
    \cite{Kashiwara:refine-Demazure}
\end{upshape}
For $\lambda \in \mathcal{P}_n$ and $ w \in S_n,$ $\displaystyle\sum_{T \in \mathcal{B}_{w}(\lambda)}\characx^{\wt(T)} = \key_{w.\lambda}$,
where for an $n$-tuple $\beta=(\beta_1,\dots,\beta_n)$, $\characx^{\beta}=x_1^{\beta_1}\cdots x_n^{\beta_n}$.
\end{proposition}
A flag $\Phi = (\Phi_1, \Phi_2, \ldots)$ is a finite weakly increasing sequence of positive integers. For $n \in \mathbb{N},$ $\mathcal{F}[n] $ denotes the set of all flags $\Phi = (\Phi _1, \Phi _2, \ldots, \Phi_n)$ such that $\Phi_n = n$. Let $\Phi \in \mathcal{F}[n]$. Then we say a set-valued reverse plane partition $H$ of shape $\lambda/\mu$ $(\lambda,\mu \in \mathcal{P}_n)$ respects flag $\Phi$ if the entries in $i^{th}$ row of $H$ are at most $\Phi_i$ $ \forall  i \in [n]$. 
\subsection{The flagged hybrid Grothendieck polynomials} Following \cite[\S6.3]{Hybrid},
we define the \emph{flagged hybrid Grothendieck polynomial} $H_{\lambda/\mu}(\characx_{\Phi}; \mathbf{t};\mathbf{w})$ by 
$$ H_{\lambda/\mu}(\characx_{\Phi}; \mathbf{t}; \mathbf{w}):= \displaystyle \sum _{T \in \SVRPP(\lambda/\mu,\Phi)} \mathbf{t}^{\ceq(T)}\mathbf{w}^{\ex(T)}\characx^{\wt(T)} , $$
where $ \SVRPP(\lambda/\mu,\Phi)$ denotes the set of all set-valued reverse plane partitions of shape $\lambda/\mu$ that respects the flag $\Phi$. Also, for $\beta =(\beta_1,\beta_2,\dots, \beta_n) \in \mathbb{Z}^n _{+}$, $ \mathbf{t}^{\beta}=t_1^{\beta_1}t_2^{\beta_2}\cdots t_n^{\beta_n} $ and $\mathbf{w}^{\beta}=w_1^{\beta_1}w_2^{\beta_2}\cdots w_n^{\beta_n} $. 

Let $\mathbf{0}=(0,0,\ldots,0) \in \mathbb{Z}^n _{+}$. We now highlight several notable specializations:
\begin{itemize}
    \item When $\Phi=(n,\dots,n)$, the polynomial $H_{\lambda/\mu}(\characx_{\Phi}; \mathbf{t};\mathbf{w}) $ is the hybrid Grothendieck polynomial $H_{\lambda/\mu}(\characx; \mathbf{t};\mathbf{w})$, which is introduced in \cite[\S1]{Hybrid} as follows: 
    $$ H_{\lambda/\mu}(\characx; \mathbf{t};\mathbf{w}):= \displaystyle \sum _{T \in \SVRPP_n(\lambda/\mu)} \mathbf{t}^{\ceq(T)}\mathbf{w}^{\ex(T)}\characx^{\wt(T)} , $$
    where $ \SVRPP_n(\lambda/\mu)$ denotes the set of all set-valued reverse plan partitions of shape $\lambda/\mu$ with entries $\leq n$.
    \item At $\mathbf{t}=\mathbf{0}$, $H_{\lambda/\mu}(\characx_{\Phi}; \mathbf{0};-\mathbf{w}) $ is the flagged refined skew stable Grothendieck polynomial $G_{\lambda/\mu}(X_\Phi; \mathbf{w})$ in \cite[\S2.3]{Sidhu:SVT}, defined by 
    $$G_{\lambda/\mu}(X_\Phi; \mathbf{w}):=\displaystyle\sum_{T \in \SVT(\lambda/\mu,\Phi)} (-1)^{|\ex(T)|} \mathbf{w}^{\ex(T)} \characx^{\wt(T)},$$ where the set of all semi-standard set-valued tableaux in $\SVRPP(\lambda/\mu,\Phi)$ is denoted by $\SVT(\lambda/\mu,\Phi)$.
    
    Also, whenever $\Phi=(n,n,\dots,n) \in \mathbb{Z}^n _{+}$, $H_{\lambda/\mu}(\characx_{\Phi}; \mathbf{0};-\mathbf{w}) $ is the row-refined skew stable Grothendieck polynomial $RG_{\lambda/\mu}(\characx;\mathbf{w}),$ see \cite[\S3]{Melody-Nathan}. Furthermore, when $\mathbf{w}=\mathbf{0}$, it coincides with the skew stable Grothendieck polynomial $G_{\lambda/\mu}(\characx)$ in \cite[\S2]{Buch-KLR}. 
    \item Setting $\mathbf{w}=\mathbf{0}$, $H_{\lambda/\mu}(\characx_{\Phi}; \mathbf{t};\mathbf{0})$ is the row-flagged refined dual stable Grothendieck polynomial $\Tilde{g}^{\row(\mathbf{1},\Phi)}_{\lambda}(\characx;\mathbf{t})$ in \cite[\S2.2]{Kim}, which is defined by
    $$ \Tilde{g}^{\row(\mathbf{1},\Phi)}_{\lambda}(\characx;\mathbf{t}):=\displaystyle\sum_{T \in \RPP(\lambda/\mu,\Phi)}  \mathbf{t}^{\ceq(T)} \characx^{\wt(T)},$$
     where $\RPP(\lambda/\mu,\Phi)$ is the set of all reverse plane partitions in $\SVRPP(\lambda/\mu,\Phi)$.
     
     Specializing further $\Phi=(n,n,\dots,n) \in \mathbb{Z}^n _{+}$, $H_{\lambda/\mu}(\characx_{\Phi}; \mathbf{t};\mathbf{0})$ reduces to the refined dual stable Grothendieck polynomial $\Tilde{g}_{\lambda/\mu}(\characx;\mathbf{t})$ \cite[\S3]{Galashin:Bender-Knuth}. Furthermore, setting $\mathbf{t}=\mathbf{0}$, it reduces to the skew dual stable Grothendieck polynomial $g_{\lambda/\mu}(\characx)$ in \cite[\S9]{Lam}. 
     \item Setting both $\mathbf{t}=\mathbf{0}, \mathbf{w}=\mathbf{0}$, $H_{\lambda/\mu}(\characx_{\Phi}; \mathbf{0};\mathbf{0}) $ is the flagged skew Schur polynomial \cite{RS}
     $s_{\lambda/\mu}(X_\Phi):=\displaystyle\sum_{T \in \Tab(\lambda/\mu,\Phi)}\characx^{\wt(T)}$, where $\Tab(\lambda/\mu,\Phi)$ is the set of all semi-standard Young tableaux in $\SVRPP(\lambda/\mu,\Phi)$. 
     
     In addition, when $\Phi=(n,n,\ldots, n) \in \mathbb{Z}^n _{+} $, $H_{\lambda/\mu}(\characx_{\Phi}; \mathbf{0};\mathbf{0})$ is the classical skew Schur polynomial $s_{\lambda/\mu}(\characx)$. 
\end{itemize}
\section{proof of the main theorem}
\label{Section 3}
In this section, we provide a Demazure crystal structure on $\SVRPP(\lambda/\mu,\Phi)$, and consequently deduce an expansion of $H_{\lambda/\mu}(\characx_{\Phi};\mathbf{t};\mathbf{w})$ in terms of key polynomials.

Fix positive integers $m$ and $n$, and let $\operatorname{Mat}_{m \times n}(\mathbb{Z}_{+})$ denote the set of all $m \times n$ matrices with non-negative integer entries. To every matrix $A = (a_{ij}) \in \operatorname{Mat}_{m \times n}(\mathbb{Z}_{+})$, we associate a biword $w_A$ represented as a two-line array:
\begin{equation}
\label{eq:biword}
    w_A = \begin{bmatrix}
    i_t & \cdots & i_2 & i_1 \\
    j_t & \cdots & j_2 & j_1
\end{bmatrix}
\end{equation}
For each pair $(i, j) \in [m] \times [n]$, the column vector $\begin{bmatrix} i \\ j \end{bmatrix}$ appears in $w_A$ with multiplicity precisely equal to $a_{ij}$. The columns are ordered according to the following conditions:
\begin{enumerate}[label=\textup{(\arabic*)}]
    \item $i_t \ge \cdots \ge i_2 \ge i_1 \ge 1$,
    \item $i_{k+1} > i_k$ whenever $j_{k+1} > j_k$.
\end{enumerate}
\begin{example}
If we assume
$ A=
\begin{pmatrix}
1 & 2 & 0 \\
0 & 1 & 3 \\
1 & 0 & 2
\end{pmatrix}
$, then $w_A =
\begin{bmatrix}
3 & 3 & 3 & 2 & 2 & 2 & 2 & 1 & 1 &1\\
1 & 3 & 3 & 2 & 3 & 3 & 3 & 1 & 2 &2 
\end{bmatrix}$.   
\end{example}
The \emph{Knuth} equivalence $\sim$ is defined on the set of all words over $\mathbb{N}$ by the transitive
closure of the relations
$$ \mathbf{u}xzy\mathbf{v} \sim \mathbf{u}zxy\mathbf{v} \quad \text{ for } x \leq y <z,$$
$$ \mathbf{u}yxz\mathbf{v} \sim \mathbf{u}yzx\mathbf{v} \quad \text{ for } x < y \leq z,$$ where $\mathbf{u},\mathbf{v}$ are arbitrary words in $\mathbb{N}$. 
\begin{definition}
    The \emph{column reading word} of a semi-standard Young tableau $T$, denoted by $\cw(T)$, is formed by reading the entries column by column  from left to right, listing the elements within each column from bottom to top.   
\end{definition}
\begin{example}
Let $T= \ytableausetup{mathmode,
notabloids}
 \begin{ytableau}
    1 & 2 & 3 \\ 2 & 3 & 5 \\ 3 & 4 
\end{ytableau}$. Then $\cw(T)=321.432.53$. We have used dots to separate subwords corresponding to different columns.
\end{example}
\begin{definition}
\begin{upshape}
    \cite[\S2]{RS}
\end{upshape}
  The \emph{column insertion} of the biword $w_A$ \eqref{eq:biword} corresponding to the matrix $A$ yields a pair $(P,Q)$ of semi-standard Young tableaux, constructed inductively for $1 \leq k \leq t$ as follows:
  
  For each $1 \leq k \leq t$, let $P_k$ be the unique semi-standard Young tableau satisfying $\cw(P_k) \sim j_k\cdots j_2j_1$. We define $P=P_t$.
  
  We define $Q$ to be the semi-standard Young tableau with $\sh(Q)=\sh(P)$ such that for $1 \leq k \leq t$, its restriction $Q_k$ to the letters $i_1,i_2,\dots,i_k$ is a semi-standard Young tableau  such that $\sh(Q_k)=\sh(P_k)$.  
\end{definition}
\begin{example}
Let
$A=
\begin{pmatrix}
1 & 0 & 1 &0 \\
0 & 1 & 1 &1 \\
\end{pmatrix}
$, then $w_A =
\begin{bmatrix}
2 & 2 & 2 & 1 & 1 \\
2 & 3 & 4 & 1 & 3 
\end{bmatrix}$. Then we have
$$ 
P_1= 
\ytableausetup{mathmode,
notabloids}
 \begin{ytableau}
    3 
 \end{ytableau}
, \quad
P_2= 
\ytableausetup{mathmode,
notabloids}
 \begin{ytableau}
    1 & 3 
 \end{ytableau}
, \quad
P_3= 
\ytableausetup{mathmode,
notabloids}
 \begin{ytableau}
    1 & 3 \\ 4 
 \end{ytableau}
, \quad
P_4= 
\ytableausetup{mathmode,
notabloids}
 \begin{ytableau}
    1 & 3 \\3 & 4 
 \end{ytableau}
, \quad
P=P_5= 
\ytableausetup{mathmode,
notabloids}
 \begin{ytableau}
    1 &3 &3 \\ 2&4 
 \end{ytableau};
$$
$$ 
Q_1= 
\ytableausetup{mathmode,
notabloids}
 \begin{ytableau}
    1 
 \end{ytableau}
, \quad
Q_2= 
\ytableausetup{mathmode,
notabloids}
 \begin{ytableau}
    1 & 1 
 \end{ytableau}
, \quad
Q_3= 
\ytableausetup{mathmode,
notabloids}
 \begin{ytableau}
    1 & 1 \\ 2 
 \end{ytableau}
, \quad
Q_4= 
\ytableausetup{mathmode,
notabloids}
 \begin{ytableau}
    1 & 1 \\2 & 2 
 \end{ytableau}
, \quad
Q=Q_5= 
\ytableausetup{mathmode,
notabloids}
 \begin{ytableau}
    1 &2 &2 \\ 2&2 
 \end{ytableau}.
$$

\end{example}

\begin{theorem}
\begin{upshape}
\cite[Appendix A, Proposition 2]{Fulton:yt} 
\end{upshape}
\label{Theorem:Burge}
The column insertion provides a bijection between $\Mat_{m \times n}(\mathbb{Z}_{+})$ and   
the set of pairs $(P,Q)$ of semi-standard Young tableaux of the same shape, with entries in $[n],[m]$, respectively. If $A$ maps to $(P,Q)$ under this bijection, we write $(w_A \rightarrow \emptyset)=(P,Q)$. 
\end{theorem}
\begin{proposition}
    \begin{upshape}
        \cite[Proposition 29]{RS}
    \end{upshape}
    \label{Prop:RS}
Let $\mathbf{u}=u_1\cdots u_p$ be a word in $[n]$ and $\mathbf{w}_{[p]}$ be the word $p\cdots 21 $. Also, let $h_i$ be either $e_i$ or $f_i$. Then if
$$(
\begin{bmatrix}
    \mathbf{w}_{[p]}\\
    \mathbf{u}
\end{bmatrix} \rightarrow \emptyset) =(P,Q) ,$$ then
$$(
\begin{bmatrix}
\mathbf{w}_{[p]}\\
h_i.\mathbf{u}
\end{bmatrix} \rightarrow \emptyset) =(h_i.P,Q) .$$
\end{proposition}
For $\alpha \in \mathbb{Z}^n_{+} $, let $\text{key}(\alpha)$ denote the unique semi-standard Young tableau of shape $\alpha^{\dagger}$ and weight $\alpha$. Given a flag $\Phi \in \mathcal{F}[n],$ we define $\mathcal{W}(\alpha, \Phi)$ to be the set of all words
$$\mathbf{w}= w^{(n)} \cdots w^{(2)} w^{(1)},$$ such that
\begin{itemize}
    \item each $w^{(i)}$ is weakly increasing word and  the last letter of $w^{(i)}$ is greater than the first letter of $w^{(i-1)}$.
    \item each word $w^{(i)}$ consists of $\alpha_i$ many letters, with the letters are bounded above by $\Phi_i$.
    \item $(
    \begin{bmatrix}
    \mathbf{b}(\alpha)\\
    \mathbf{w}
    \end{bmatrix} \rightarrow \emptyset) =(-, \text{key}(\alpha))$, where $\mathbf{b}(\alpha)$ the word $ b^{(n)} \cdots b^{(2)}b^{(1)}$ in which $b^{(j)}$ consists of a string of $\alpha_j$ copies of $j$. For instance, $\mathbf{b}(1,2,0,2) = 44221$.
\end{itemize}
\begin{example}
    Let $\alpha=(1,2,0,1)$ and $ \Phi=(1,2,3,4)$. Then the set of all words $\mathbf{w}= w^{(4)} w^{(3)} w^{(2)} w^{(1)}$ in $\mathcal{W}(\alpha, \Phi)$ are given below:
    $$ 3\cdot\cdot 12 \cdot 1 \quad 3\cdot \cdot 22 \cdot 1 \quad 4 \cdot \cdot 22 \cdot 1 .$$
\end{example}
\begin{definition}
\begin{upshape}
    \cite{RS}
\end{upshape} 
\label{def:ogct}
A semi-standard Young tableau $Q$ is called $(\lambda/\mu, \Phi)$-compatible if there exists a unique semi-standard Young tableau $Q_0 \in \Tab(\lambda/\mu,\Phi)$ such that
\begin{enumerate}
    \item $\Read(Q_0)$ is a Yamanouchi word.
    \item Under the column insertion (Theorem~\ref{Theorem:Burge}), the biword, $ \begin{bmatrix}
        \mathbf{b}(\lambda/\mu) \\
        \Read(Q_0) \\
\end{bmatrix} $ 
maps to $ (T_{\sh(Q)}, Q)$, where $\mathbf{b}(\lambda/\mu)$ denotes the word $\mathbf{b}(\lambda-\mu)$.
\end{enumerate}
\end{definition}

For any $(\lambda/\mu, \Phi)$-compatible tableau $Q$, following \cite[Appendix]{KRSV}, we define
$$\mathcal{A}(Q, \lambda/\mu,\Phi):=\{ T \in \Tab(\lambda/\mu,\Phi) :  
(\begin{bmatrix}
           \mathbf{b}(\lambda/\mu) \\
           \Read(T) \\
\end{bmatrix} \rightarrow \emptyset) = (\rect (T), Q)\},$$
where $\rect(T)$ denotes the unique semi-standard Young tableau Knuth equivalent to $\Read(T)$.

Consequently, we have the disjoint union decomposition: $$\Tab(\lambda/\mu,\Phi)=\displaystyle\bigsqcup_{Q}\mathcal{A}(Q, \lambda/\mu,\Phi),$$ where $Q$ runs over all $(\lambda/\mu, \Phi)$-compatible tableaux.

For a $(\lambda/\mu, \Phi)$-compatible tableau $Q$, 
let $\beta(Q)$ denote the weight of the left key tableau $K_{\_}(Q)$ \cite{Willis:left-key,Mrigendra:left-key} of $Q$. Then we have the following proposition  
\begin{proposition}
\begin{upshape}
    \cite[Proposition A.7]{KRSV}
\end{upshape}
\label{Prop:Alco}
There exists 
a bijection
$$\Omega: \mathcal{A}(Q, \lambda/\mu, \Phi) \rightarrow \mathcal{W}(\beta(Q), \Phi)$$ such that, if $ T \mapsto \Omega (T)$, then $\Read(T)$ and $\Omega{(T)}$ are Knuth equivalent.    
\end{proposition}
We now recall the following two important theorems.
\begin{theorem}
\begin{upshape} \cite[Theorem 21]{RS} \end{upshape}
\label{theorem:RS-hat}
Let $\beta \in \mathbb{Z}_+^{n}$, $\Phi \in \mathcal{F}[n]$ and $\Phi_0$ denotes the standard flag $(1,2,\dots,n)$. Then if $\mathcal{W}(\beta, \Phi)$ is non-empty, then there exists $\widehat{\beta} \in \mathbb{Z}_+^{n}$ satisfyig $\beta^{\dagger} = \widehat{\beta}^{\dagger}$ and a bijection 
$$\zeta:\mathcal{W}(\beta, \Phi) \rightarrow  \mathcal{W}(\widehat{\beta}, \Phi_0) ,$$ such that if $\mathbf{u} \mapsto \zeta (\mathbf{u})$ then $\mathbf{u}$ and $\zeta (\mathbf{u})$ are Knuth equivalent. 
\end{theorem}
\begin{theorem}
\begin{upshape} \cite[Proposition 5.6]{LS-keys} \end{upshape}
Let $\alpha = \tau.\alpha ^{\dagger}$. Then the map sending $\mathbf{u} \in \mathcal{W}(\alpha, \Phi_0)$ to the unique semi-standard Young tableau $ P(\mathbf{u})$ Knuth equivalent to $\mathbf{u}$, induces a bijection between $ \mathcal{W}(\alpha, \Phi_0)$ and the Demazure crystal $\mathcal{B}_{\tau}(\alpha^{\dagger})$.
\end{theorem}
Then combining the above two theorems with Proposition~\ref{Prop:Alco}, we have a Demazure crystal structure for $\Tab(\lambda/\mu,\Phi)$, due to the following propostion.
\begin{proposition}
\begin{upshape}
    \cite[Proposition A.9]{KRSV}
\end{upshape}
\label{proposition:tableau}
The rectification map $ \rect: \mathcal{A}(Q, \lambda/\mu,\Phi) \rightarrow  \mathcal{B}_{\tau}(\widehat{\beta(Q)}^{\dagger})$ is a weight-preserving bijection that commutes with $e_i,f_i$, for each $i$, where $\tau$ is any permutation satisfying $\tau. \widehat{\beta(Q)}^{\dagger}=\widehat{\beta(Q)}$. In 
other words,
$$\Tab(\lambda/\mu,\Phi) \cong \displaystyle\bigsqcup_{Q} \mathcal{B}_{\tau} (\widehat{\beta(Q)}^{\dagger}.$$,
\end{proposition}
Given two skew shapes $\theta, \theta'$, we denote $\theta * \theta'$ by their corner-to-corner concatenation as shown below:
$$
\begin{tikzpicture}[scale=0.9]
  \draw (0,0)--(1.4,0)--(1.4,1.4);
  \draw (0,0)--(0,0.7)--(1.4,0.7);
  \draw (0.7,0)--(0.7,1.4)--(1.4,1.4);
  \node at (-0.6,0.7) {$\theta=$};
  \node at (0,-0.35) {$\null$};
\end{tikzpicture}
\begin{tikzpicture}
    \node at (0,0) {$\null$};
    \node at (0.5,0) {$\null$}; 
\end{tikzpicture}
\begin{tikzpicture}[scale=0.9]
    \draw (0,0) -- (0,1.4)--(1.4,1.4)--(1.4,0.7)--(0.7,0.7)--(0.7,0)--(0,0);
    \draw (0,0.7)--(0.7,0.7)--(0.7,1.4);
    \node at (-0.6,0.7) {$\theta'=$};
    \node at (0,-0.35) {$\null$};
\end{tikzpicture}
\begin{tikzpicture}
    \node at (0,0) {$\null$};
    \node at (0.5,0.5) {$\implies$};
    \node at (0,-0.35) {$\null$};
\end{tikzpicture}
\begin{tikzpicture}
    \draw (0,0)--(1.4,0)--(1.4,1.4);
    \draw (0,0)--(0,0.7)--(1.4,0.7);
    \draw (0.7,0)--(0.7,1.4)--(1.4,1.4);
    \draw (0+1.4,0+1.4) -- (0+1.4,1.4+1.4)--(1.4+1.4,1.4+1.4)--(1.4+1.4,0.7+1.4)--(0.7+1.4,0.7+1.4)--(0.7+1.4,0+1.4)--(0+1.4,0+1.4);
    \draw (0+1.4,0.7+1.4)--(0.7+1.4,0.7+1.4)--(0.7+1.4,1.4+1.4);
    \node at (-1,0.7+0.35) {$\theta * \theta'=$};
\end{tikzpicture}.
$$
Similarly, if $T,T'$ are semi-standard Young tableaux of shapes $\theta, \theta'$ respectively then we define another semi-standard Young tableau $T*T'$ of skew shape $\theta*\theta'$ through corner-to-corner concatenation. For skew shapes $ \theta_1,\theta_2,\theta_3$, we write $\theta_3 * \theta_2 *\theta_1$ to denote $\theta_3 * (\theta_2 * \theta_1)$. Likewise, given semi-standard Young tableaux $T_i$ of shape $\theta_i $ for $i=1,2,3$, then we write $T_3*T_2*T_1$ to denote $T_3*(T_2*T_1)$. 

For a skew shape $\lambda/\mu$ with $\lambda,\mu \in \mathcal{P}_n$ and $\alpha,\beta \in \mathbb{Z}_{+}^{n}$, we set
$$[\lambda/\mu]^{\alpha}_{\beta}:=(\lambda_n-\mu_n-\alpha_{n-1},1^{\beta_n}) *\cdots * (\lambda_2-\mu_2-\alpha_1,1^{\beta_2}) * (\lambda_1-\mu_1,1^{\beta_1}).$$
Then each set-valued reverse plane partition $H$ of shape $\lambda/\mu$ satisfying $\ceq(H)=\alpha$ and $\ex(H)= \beta$ uniquely induces to a semi-standard Young tableau $\Tilde{H} =\Tilde{H}_n * \cdots *\Tilde{H}_2 *\Tilde{H}_1$ of shape $[\lambda/\mu]^{\alpha}_{\beta}$, where each $\Tilde{H}_i$ is the semi-standard Young tableau such that for $i \geq 1$ 
\begin{itemize}
    \item $\sh(\Tilde{H}_i)=(\lambda_i-\mu_i-\alpha_{i-1}, 1^{\beta_i})$. If $\lambda_j-\mu_j=\alpha_{j-1}$ for some $j$ then $\sh(\Tilde{H}_j)=(1^{\beta_j})$. Here we assume $\alpha_0=0$.
    \item $\Read(\Tilde{H}_i)$ is subword of $\Read(H)$ that comes from the $i^{th}$ row of $H$.
\end{itemize}
\begin{example}
Let us assume
$$
\begin{tikzpicture}[scale=1.2]
  \draw (0,0)--(1.6,0)--(1.6,1.6)--(2.4,1.6);
  \draw (0,0)--(0,0.8)--(2.4,0.8)--(2.4,1.6);
  \draw (0.8,0)--(0.8,1.6)--(1.6,1.6);
  \draw (1.6,0)--(2.4,0)--(2.4,0.8);
  \node at (0.4,0.4) {$1,\textcolor{blue}{2}$};
  \node at (1.2,0.4) {$\textcolor{red}{2},\textcolor{blue}{3}$};
  \node at (2,0.4) {$\textcolor{red}{4},\textcolor{blue}{5}$};
  \node at (1.2,1.23) {$2$};
  \node at (2,1.2) {$2,\textcolor{blue}{3,4}$};
  \node at (-0.4,0.8) {$H=$};
\end{tikzpicture}.
$$
In this case, we have $\Read(H) = \textcolor{blue}{532}.1|\textcolor{blue}{43}.22$, which yields $\Read(\Tilde{H}_1)=4332, \Read(\Tilde{H}_2)= 5321$.
Therefore,
$$ \ytableausetup{mathmode,
notabloids}
\Tilde{H}_1= \begin{ytableau}
     2 &2 \\ \textcolor{blue}{3} \\ \textcolor{blue}{4}  
\end{ytableau}
\text{\hspace{0.5 cm} and \hspace{0.5 cm}}
\ytableausetup{mathmode,
notabloids}
\Tilde{H}_2=\begin{ytableau}
      1 \\ \textcolor{blue}{2} \\ \textcolor{blue}{3} \\\textcolor{blue}{5}
\end{ytableau}.$$
\end{example}
We define
$$\SVRPP^{\alpha}_{\beta}(\lambda/\mu, \Phi):= \{ H \in \SVRPP(\lambda/\mu, \Phi): \ceq(H)=\alpha, \ex(H)=\beta \}.$$
Then if $H \in \SVRPP^{\alpha}_{\beta}(\lambda/\mu, \Phi)$ and $\Tilde{H}=\Tilde{H}_n*\cdots * \Tilde{H}_2 * \Tilde{H}_1$ then $\Tilde{H}_i$ respects the flag $(\underbrace{\phi_i,\dots,\phi_i}_{\beta_i+1})$ for $1 \leq i \leq n$. Thus $\Tilde{H}\in \Tab([\lambda/\mu]^{\alpha}_{\beta},\Phi^{\beta}),$ where $$\Phi^{\beta}:=(\underbrace{\phi_1,\dots,\phi_1}_{\beta_1+1},\underbrace{\phi_2,\dots,\phi_2}_{\beta_2+1},\dots,\underbrace{\phi_n,\dots,\phi_n}_{\beta_n+1} ).$$
\begin{definition}
\label{def:comp}
Let $Q$ be a highest weight element in $ \SVRPP^{\alpha}_{\beta}(\lambda/\mu, \Phi)$. Then $\Tilde{Q}=\Tilde{Q}_n * \cdots * \Tilde{Q}_2 *\Tilde{Q}_1$ and $\sh(\Tilde{Q})=[\lambda/\mu]^{\alpha}_{\beta}$. Now if the column insertion (Theorem~\ref{Theorem:Burge}) maps the biword $
\begin{bmatrix}
 \mathbf{b}\big([\lambda/\mu]^{\alpha}_{\beta}\big) \\
 \Read(\Tilde{Q}) \\
\end{bmatrix}$ to the pair $(\rect (\Tilde{Q}), \hat{Q})$, then we say $\hat{Q}$ is a $(\lambda/\mu,\Phi)$-compatible tableau for SVRPP.

We also define $\overline{\ceq}(\hat{Q}):=\ceq(Q)=\alpha$ and $\overline{\ex}(\hat{Q}):=\ex(Q)=\beta$. We denote $\CT^{\alpha}_{\beta}(\lambda/\mu,\Phi)$ by the set of all $(\lambda/\mu,\Phi)$-compatible tableaux for SVRPP $\hat{Q}$ such that $\overline{\ceq}(\hat{Q})=\alpha$ and $\overline{\ex}(\hat{Q})=\beta$. 
\end{definition}
\bremark
If $\hat{Q}\in \CT^{\mathbf{0}}_{\mathbf{0}}(\lambda/\mu,\Phi)$ then $\hat{Q}$ is a $(\lambda/\mu,\Phi)$-compatible tableau, as mentioned in Definition~\ref{def:ogct}. 
\eremark
\begin{example}
Let $\lambda = (4,2,2), \mu =(2,1,0), \alpha=(1,1,0), \beta=(1,2,1), \Phi=(2,3,4)$. Also, we take
$$ Q= \ytableausetup{mathmode,
notabloids}
\begin{ytableau}
    \none & \none &1 &1,\textcolor{blue}{2}\\ \none & 1,\textcolor{blue}{2} & 2 & \textcolor{red}{2},\textcolor{blue}{3}\\ 1 &\textcolor{red}{2},\textcolor{blue}{4}  
\end{ytableau} \in \SVRPP^{\alpha}_{\beta}(\lambda/\mu, \Phi).
$$
Thus $\Read(Q)=\textcolor{blue}{4}1\textcolor{blue}{32}12\textcolor{blue}{2}11$ and it is easy to see that $\Read(Q)$ is a Yamanouchi word. So by Proposition~\ref{Prop:Yama}, $Q$ is a highest weight element in $\SVRPP^{\alpha}_{\beta}(\lambda/\mu, \Phi)$. Also, $\Read(\Tilde{Q}_1)=211, \Read(\Tilde{Q}_2) = 3212$ and $ \Read(\Tilde{Q}_3)=41$. Hence $\Tilde{Q}=\Tilde{Q}_3 * \Tilde{Q}_2 * \Tilde{Q}_1$, where 
$$
\Tilde{Q}_1= \ytableausetup{mathmode,
notabloids}
\begin{ytableau}
 1 & 1 \\2
\end{ytableau} \hspace{1 cm} 
\Tilde{Q}_2= \ytableausetup{mathmode,
notabloids}
\begin{ytableau}
    1 &2 \\ 2 \\3 
\end{ytableau} \hspace{1 cm}
\Tilde{Q}_3= \ytableausetup{mathmode,
notabloids}
\begin{ytableau}
    1 \\4   
\end{ytableau}.$$
Now, using Theorem~\ref{Theorem:Burge}, we obtain
$$\Big(
\begin{bmatrix}
    \mathbf{b}\big([\lambda/\mu]^{\alpha}_{\beta}\big) \\
    \Read(\Tilde{Q}) \\
\end{bmatrix} \rightarrow \emptyset \Big)
=(Q',\hat{Q}),$$
where 
$$Q'=\ytableausetup{mathmode,
notabloids}
  \begin{ytableau}
    1 &1 &1 &1 \\2 &2 &2 \\3 \\4  
  \end{ytableau}=\rect(\Tilde{P}) 
\text{ and }  
\hat{Q}
=\ytableausetup{mathmode,
notabloids}
  \begin{ytableau}
    1 &1 &3 &6\\2 &3 &4 \\ 5 \\7
  \end{ytableau}.$$
So $\hat{Q} = \ytableausetup{mathmode,
notabloids}
\begin{ytableau}
    1 &1 &3 &6\\2 &3 &4 \\ 5 \\7
\end{ytableau}$
is a $(\lambda/\mu, \Phi)$-compatible tableau for SVRPP such that $\overline{\ceq}(\hat{Q})=\ceq(Q)=(1,1,0)$ and $\overline{\ex}(\hat{Q})=\ex(Q)=(1,2,1)$.
\end{example}
For any highest weight element $ Q \in \SVRPP^{\alpha}_{\beta}(\lambda/\mu, \Phi)$, let $ \SVRPP^{\alpha}_{\beta}(\lambda/\mu, \Phi;Q)$ be the connected component of the crystal graph of $\SVRPP^{\alpha}_{\beta}(\lambda/\mu, \Phi)$ containing $Q$. Then 
\begin{equation}
\label{eq:Demazure:excess}
\SVRPP^{\alpha}_{\beta}(\lambda/\mu, \Phi) = \bigsqcup_{Q} \SVRPP^{\alpha}_{\beta}(\lambda/\mu, \Phi;Q),
\end{equation}
$Q$ runs over all highest weight elements in $\SVRPP^{\alpha}_{\beta}(\lambda/\mu, \Phi) $.
Our goal is now to prove that $\SVRPP^{\alpha}_{\beta}(\lambda/\mu, \Phi;P)$ is isomorphic to a Demazure crystal, which implies $\SVRPP^{\alpha}_{\beta}(\lambda/\mu, \Phi)$ admits a Demazure crystal structure. To this end, we define the following map 
$$\Theta : \SVRPP^{\alpha}_{\beta}(\lambda/\mu, \Phi;Q) \rightarrow \mathcal{A}(\hat{Q}, [\lambda/\mu]^{\alpha}_{\beta},\Phi^{\beta}) \text{ by }
H \mapsto \Tilde{H} =\Tilde{H}_n * \cdots * \Tilde{H}_2*\Tilde{H}_1.$$
Together with Proposition~\ref{proposition:tableau}, the following proposition endows a Demazure crystal structure on $\SVRPP^{\alpha}_{\beta}(\lambda/\mu, \Phi)$.
\begin{proposition}
\label{Proposition:main}
The map $\Theta : \SVRPP^{\alpha}_{\beta}(\lambda/\mu, \Phi;Q)  \rightarrow \mathcal{A}(\hat{Q}, [\lambda/\mu]^{\alpha}_{\beta},\Phi^{\beta})$ defined by $ \Theta(H)=\Tilde{H}$ is a weight-preserving bijection which intertwines $e_i,f_i$ for each $i$.
\end{proposition}
\begin{proof}
We first check that the map $H \mapsto \Tilde{H}$ commutes with $e_i,f_i$ for each $i$. Let $h_i $ be either $e_i$ or $f_i$. We have to show $\Theta(h_i.H)=h_i.\Theta(H)=h_i.\Tilde{H}$. It is enough to prove that $\Read(\widetilde{h_i.H})=\Read(h_i\Tilde{H})$.

It is clear that
$\Read(\widetilde{h_i.H})= \Read(h_i.H)$.
Applying Remark\ref{remark:crystal}, we have $\Read(h_i.H) = h_i.\Read(H)$. Using $\Read(H)=\Read(\Tilde{H})$ and Remark\ref{remark:crystal}, we get $h_i. \Read(\Tilde{H}) = \Read(h_i.\Tilde{H})$. Therefore, $\Read(\widetilde{h_i.H})=\Read(h_i\Tilde{H})$. 

We now verify that if $H \in \SVRPP^{\alpha}_{\beta}(\lambda/\mu, \Phi;Q)$, then $\Tilde{H} \in \mathcal{A}(\hat{Q}, [\lambda/\mu]^{\alpha}_{\beta},\Phi^{\beta})$. Let $H=f_{i_1}^{k_1}\cdots f_{i_r}^{k_r}.Q =\Bar{f}.Q$, where $\Bar{f}=f_{i_1}^{k_1} \cdots f_{i_r}^{k_r}$. Now by Definition~\ref{def:comp},
$$
\begin{bmatrix}
 \mathbf{b}\big([\lambda/\mu]^{\alpha}_{\beta}\big) \\
 \Read(\Tilde{Q}) \\
\end{bmatrix} = (\rect (\Tilde{Q}), \hat{Q}).$$
Using Proposition~\ref{Prop:RS}, we have
$$
\begin{bmatrix}
 \mathbf{b}\big([\lambda/\mu]^{\alpha}_{\beta}\big) \\
 \Bar{f}.\Read(\Tilde{Q}) \\
\end{bmatrix} = (\Bar{f}.\rect (\Tilde{Q}), \hat{Q}).$$
Applying Remark~\ref{remark:crystal}, we obtain
$$
\begin{bmatrix}
 \mathbf{b}\big([\lambda/\mu]^{\alpha}_{\beta}\big) \\
 \Read(\Bar{f}.\Tilde{Q}) \\
\end{bmatrix} = (\rect (\Bar{f}.\Tilde{Q}), \hat{Q}).$$
Since $\Tilde{H}= \Theta(H)=\Theta(\Bar{f}.Q) =\Bar{f}.\Tilde{Q}$, we have $\Tilde{H} \in \mathcal{A}(\hat{Q}, [\lambda/\mu]^{\alpha}_{\beta},\Phi^{\beta})$. It is clear that $\wt(H)=\wt(\Tilde{H})$.

Let $\Theta(S)=\Theta(T) \implies \Read(\Tilde{S})=\Read(\Tilde{T})$. Since $\Read(S)= \Read(\Tilde{S}) $ and $\Read(T) = \Read(\Tilde{T}) $, we have $\Read(S)=\Read(T) $. Also, by definition, $\ex(S)=\ex(T)=\beta$ and $\ceq(S)=\ceq(T)=\alpha$. Thus $\h(S)=\h(T)$. Therefore, by Proposition~\ref{Prop:unique}, we obtain $S=T$. So $\Theta$ is injective.

Let $\Tilde{T} \in \mathcal{A}(\hat{Q}, [\lambda/\mu]^{\alpha}_{\beta},\Phi^{\beta})$. Then $\Tilde{T}=f_{j_1}^{l_1}f_{j_2}^{l_2}\cdots f_{j_t}^{l_t}.\Tilde{Q} = f.\Tilde{Q}$, where $f=f_{j_1}^{l_1}f_{j_2}^{l_2}\cdots f_{j_t}^{l_t}$. Now $\Read(Q)=\Read(\Tilde{Q})$. Hence by Remark~\ref{remark:crystal}, $\Read(\Tilde{T})=f.\Read(\Tilde{Q})= f.\Read(Q) = \Read(f.Q)$. Let $T=f.Q$. Then by Remark\ref{remark:stat}, $\ceq(T)=\alpha$ and $\ex(T)=\beta$. Also, let $w^{(i)}(T) $ be the subword of $\Read(T)$ that comes from the $i^{th}$ row $T_i$ of $T$. Since $w^{(i)}(T)=\Read(\Tilde{T}_i) $, the maximum entry of $ w^{(i)}(T)$ is $ \leq \Phi_i$. Now let $p$ be an entry of $T_i$ but not $w_i(T)$. We show that $p \leq \Phi_i$. Suppose that $p \in T_{i,j}$, where $T_{i,j}$ is set contained in the box $(i,j) \in \lambda/\mu $. It is clear from the definition of the reading word that $p$ is the minimum value in $T_{i,j}$. Then the following cases may arise:
\begin{itemize}
    \item There is another entry $q$ in $T_{i,j}$. Then $q$ will appear in $w^{(i)}(T)$. So we have 
    $$p <q \leq \max(w^{(i)}(T)) \leq \Phi_i.$$
    \item $T_{i,j}=\{p\}$. Let $b$ be the least number in $1 \leq b \leq i-1$ such that $p \in T_{b,j}$. Then $p$ will appear in $ w^{(b)}(T)$. Thus $p \leq \Phi_b \leq \Phi_i$ (since $b <i$).
\end{itemize}
Thus all the entries of $T_i$ that are not appearing in $ w^{(i)}(T)$ are $\leq \Phi_i$ and hence $\Theta$ is surjective.
\end{proof}
\begin{example}
\label{Ex: main Pro 1}
Let $\lambda=(2,2), \mu=(1,0), \Phi=(2,3),\alpha=(1,0)$ and $\beta=(1,1)$. The only highest weight elements of $\SVRPP^{\alpha}_{\beta}(\lambda/\mu, \Phi)$ (see \cite[Appendix~A]{Hybrid}) are the following:
$$
P=\begin{ytableau}
    \none &1,2\\1 &2,3   
\end{ytableau}
\hspace{1 cm}
Q=\begin{ytableau}
    \none &1,2\\1,2&2   
\end{ytableau}.
$$
Now $\SVRPP^{\alpha}_{\beta}(\lambda/\mu, \Phi;P) =
\Bigg\{
\begin{ytableau}
    \none &1,2\\1 &2,3   
\end{ytableau}, \,
\begin{ytableau}
    \none &1,2\\2 &2,3   
\end{ytableau} \Bigg \}
$ and $\SVRPP^{\alpha}_{\beta}(\lambda/\mu, \Phi;Q)=\{ Q\}$. Then we obtain $\SVRPP^{\alpha}_{\beta}(\lambda/\mu, \Phi;P) \cong \mathcal{B}_{s_1} (2,1,1)$ and $ \SVRPP^{\alpha}_{\beta}(\lambda/\mu, \Phi;Q) \cong \mathcal{B}_{id} (2,2,0)$, see Figure~\ref{fig:P}.  
\begin{figure}
    \centering\begin{tikzpicture}
    \draw (0,0)--(1.6,0)--(1.6,1.6);
    \draw (0,0)--(0,0.8)--(1.6,0.8);
    \draw (0.8,0)--(0.8,1.6)--(1.6,1.6);
    \node at (0.4,0.4+0.035) {$1$};
    \node at (1.2,0.4) {$\textcolor{red}{2},3$};
    \node at (1.2,1.2) {$1,2$};
    \draw[blue][->] (0.8,-0.25) -- (0.8,-1.25);
    \node at (0.8+0.2,-0.8) {$\textcolor{blue}{1}$};
    \draw (0,0-3.3)--(1.6,0-3.3)--(1.6,1.6-3.3);
    \draw (0,0-3.3)--(0,0.8-3.3)--(1.6,0.8-3.3);
    \draw (0.8,0-3.3)--(0.8,1.6-3.3)--(1.6,1.6-3.3);
    \node at (0.4,0.4-3.3+0.035) {$2$};
    \node at (1.2,0.4-3.3) {$\textcolor{red}{2},3$};
    \node at (1.2,1.2-3.3) {$1,2$};
\end{tikzpicture}
\begin{tikzpicture}
    \draw (0,0)  node {$\null$};
    \draw[|->] (0.5,2.5) -- (1.5,2.5);
    \draw (1,2.8) node {$\Theta$};
    \draw (1.5+0.3,3.5) node {$\null$};    
\end{tikzpicture}
\begin{tikzpicture}[scale=0.8]
    \draw (0,0)--(0,1.6)--(0.8,1.6)--(0.8,3.2)--(1.6,3.2)--(1.6,1.6)--(0.8,1.6)--(0.8,0)--(0,0);
    \draw (0,0.8)--(0.8,0.8);
    \draw (0.8,2.4)--(1.6,2.4);
    \node at (0.4,0.4) {$3$};
    \node at (0.4,1.2) {$1$};
    \node at (1.2,2) {$2$};
    \node at (1.2,2.8) {$1$};
    \draw[blue][->] (0.4,-0.25) -- (0.4,-1.25);
    \node at (0.4+0.2,-0.8) {$\textcolor{blue}{1}$};
    \draw (0,0-4.7)--(0,1.6-4.7)--(0.8,1.6-4.7)--(0.8,3.2-4.7)--(1.6,3.2-4.7)--(1.6,1.6-4.7)--(0.8,1.6-4.7)--(0.8,0-4.7)--(0,0-4.7);
    \draw (0,0.8-4.7)--(0.8,0.8-4.7);
    \draw (0.8,2.4-4.7)--(1.6,2.4-4.7);
    \node at (0.4,0.4-4.7) {$3$};
    \node at (0.4,1.2-4.7) {$2$};
    \node at (1.2,2-4.7) {$2$};
    \node at (1.2,2.8-4.7) {$1$};
\end{tikzpicture}
\begin{tikzpicture}
    \draw (0,0)  node {$\null$};
    \draw[|->] (0.5,2.5) -- (1.5,2.5);
    \draw (1,2.8) node {$\rect$};
    \draw (1.5+0.3,2.5) node {$\null$};
\end{tikzpicture}
\begin{tikzpicture}[scale=0.8]
    \draw (0,0)--(0,2.4)--(1.6,2.4)--(1.6,1.6)--(0.8,1.6)--(0.8,0)--(0,0);
    \draw (0,0.8)--(0.8,0.8);
    \draw (0,1.6)--(0.8,1.6)--(0.8,2.4);
    \node at (0.4,0.4) {$3$};
    \node at (0.4,1.2) {$2$};
    \node at (0.4,2) {$1$};
    \node at (1.2,2) {$1$};
    \draw[blue][->] (0.4,-0.25) -- (0.4,-1.35);
    \node at (0.7,-0.8) {$\textcolor{blue}{1}$};
    \draw (0,0-4)--(0,2.4-4)--(1.6,2.4-4)--(1.6,1.6-4)--(0.8,1.6-4)--(0.8,0-4)--(0,0-4);
    \draw (0,0.8-4)--(0.8,0.8-4);
    \draw (0,1.6-4)--(0.8,1.6-4)--(0.8,2.4-4);
    \node at (0.4,0.4-4) {$3$};
    \node at (0.4,1.2-4) {$2$};
    \node at (0.4,2-4) {$1$};
    \node at (1.2,2-4) {$2$};
\end{tikzpicture}

\begin{tikzpicture}
    \draw (0,0)--(1.6,0)--(1.6,1.6);
    \draw (0,0)--(0,0.8)--(1.6,0.8);
    \draw (0.8,0)--(0.8,1.6)--(1.6,1.6);
    \node at (0.4,0.4) {$1,2$};
    \node at (1.2,0.4+0.035) {$\textcolor{red}{2}$};
    \node at (1.2,1.2) {$1,2$};
\end{tikzpicture}
\begin{tikzpicture}
    \draw (0,0)  node {$\null$};
    \draw[|->] (0.5,1) -- (1.5,1);
    \draw (1,1.2) node {$\Theta$};
    \draw (1.5+0.3,1) node {$\null$};    
\end{tikzpicture}
\begin{tikzpicture}[scale=0.8]
    \draw (0,0)--(0,1.6)--(0.8,1.6)--(0.8,3.2)--(1.6,3.2)--(1.6,1.6)--(0.8,1.6)--(0.8,0)--(0,0);
    \draw (0,0.8)--(0.8,0.8);
    \draw (0.8,2.4)--(1.6,2.4);
    \node at (0.4,0.4) {$2$};
    \node at (0.4,1.2) {$1$};
    \node at (1.2,2) {$2$};
    \node at (1.2,2.8) {$1$};
    \node at (1.2,3.5) {$\null$};
\end{tikzpicture}
\begin{tikzpicture}
    \draw (0,0)  node {$\null$};
    \draw[|->] (0.4,1) -- (1.4,1);
    \draw (0.9,1.3) node {$\rect$};
    \draw (1.5,1) node {$\null$};
\end{tikzpicture}
\begin{tikzpicture}[scale=0.8]
    \draw (0,0)--(0,1.6)--(1.6,1.6)--(1.6,0)--(0,0);
    \draw (0,0.8)--(1.6,0.8);
    \draw (0.8,0)--(0.8,1.6);
    \node at (0.4,0.4) {$2$};
    \node at (1.2,0.4) {$2$};
    \node at (0.4,1.2) {$1$};
    \node at (1.2,1.2) {$1$};
\end{tikzpicture}
\caption{Decomposition of $\SVRPP^{(1,0)}_{(1,1)}\Big((2,2)/(1),(1,3)\Big)$ into Demazure crystals}
\label{fig:P}
\end{figure}
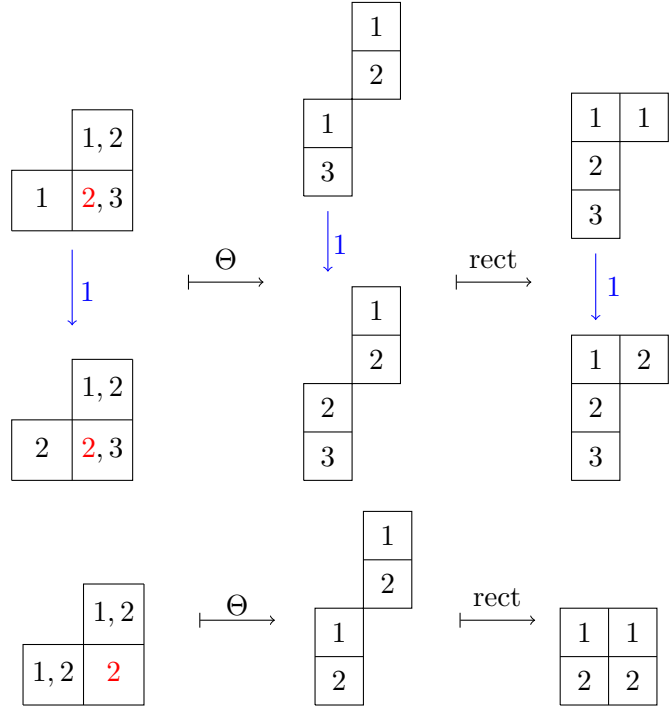
\end{example}
\bremark
In general, $\displaystyle\sum_{H \in \SVRPP^{\alpha}_{\beta}(\lambda/\mu,\Phi)}\characx^{\wt(H)}$ differs from the flagged skew Schur polynomial $s_{[\lambda/\mu]^{\alpha}_{\beta}}(X_{\Phi^{\beta}})$. For example, if we assume $\lambda/\mu$ as in Example~\ref{Ex: main Pro 1} and $ \Phi=(1,3), \alpha=(1,0), \beta=(0,1)$, then
$\SVRPP^{\alpha}_{\beta}(\lambda/\mu, \Phi) =
\Bigg\{
\begin{ytableau}
    \none &1\\1 &\textcolor{red}{1},2   
\end{ytableau}, \,
\begin{ytableau}
    \none &1\\1 &\textcolor{red}{1},3   
\end{ytableau} \Bigg \}.$
Thus we have
$$ \displaystyle\sum_{H \in \SVRPP^{\alpha}_{\beta}(\lambda/\mu,\Phi)}\characx^{\wt(H)} =x_1^2x_2 +x_1^2x_3= \kappa_{(2,0,1)}.$$
Also, $\Tab([\lambda/\mu]^{\alpha}_{\beta},\Phi^{\beta}) $ contains the following tableaux:
$$ 
\begin{ytableau}
    \none &1\\1 \\2   
\end{ytableau}, \quad 
\begin{ytableau}
    \none &1\\1 \\3   
\end{ytableau}, \quad 
\begin{ytableau}
    \none &1\\2 \\3   
\end{ytableau}.
$$
Hence we obtain
$s_{[\lambda/\mu]^{\alpha}_{\beta}}(X_{\Phi^{\beta}})= x_1^2x_2 +x_1^2x_3 +x_1x_2x_3 =\kappa_{(2,0,1)}+\kappa_{(1,1,1)}$.
\eremark
Now using \eqref{eq:Demazure:excess}, Proposition~\ref{proposition:tableau} and Proposition~\ref{Proposition:main} we have the following theorem. 
\begin{theorem}
\label{theorem:main}
Fix $\alpha , \beta \in \mathbb{Z}_+^{n}$. Also, let $\lambda/\mu$ be a skew shape with $\lambda, \mu \in \mathcal{P}_n$ and $\Phi \in \mathcal{F}[n]$. Then 
$$ \SVRPP^{\alpha}_{\beta}(\lambda/\mu, \Phi)  \cong \displaystyle\bigsqcup_{Q \in \CT^{\alpha}_{\beta} (\lambda/\mu,\Phi)  } \mathcal{B}_{\tau} (\widehat{\beta(Q)}^{\dagger}.$$
\end{theorem}
\bremark
\label{remark:main}
Theorem~\ref{theorem:main} gives a Demazure crystal structure on $\SVRPP(\lambda/\mu, \Phi)$ since $$\SVRPP(\lambda/\mu, \Phi)=\displaystyle\bigsqcup_{\alpha, \beta \in \mathbb{Z}^n _{+}} \SVRPP^{\alpha}_{\beta}(\lambda/\mu, \Phi) .$$
\eremark
\begin{corollary}
\label{corollary:SVRPP}
Fix $\alpha,\beta \in \mathbb{Z}_+^n$, $\Phi \in \mathcal{F}[n]$ and a skew shape $\lambda/\mu$ with $\lambda, \mu \in \mathcal{P}_n$. Then we have 
\begin{enumerate}
    \item Let $\SVT_{\beta}(\lambda/\mu, \Phi)$ be the set of all semi-standard set-valued tableaux $S \in \SVT(\lambda/\mu,\Phi)$ such that $\ex(S)=\beta$. Therefore $\SVT(\lambda/\mu, \Phi)=\displaystyle\bigsqcup_{ \beta \in \mathbb{Z}^n _{+}} \SVT_{\beta}(\lambda/\mu, \Phi) $. Since $\SVT_{\beta}(\lambda/\mu, \Phi)=\SVRPP^{\mathbf{0}}_{\beta}(\lambda/\mu, \Phi)$, by Theorem~\ref{theorem:main}, we get the Demazure crystal construction on $\SVT(\lambda/\mu, \Phi) $ described in \cite[Remark~3]{Sidhu:SVT}
    $$ \SVT(\lambda/\mu, \Phi)  \cong \displaystyle\bigsqcup_{\beta \in \mathbb{Z}^n _{+}} \displaystyle\bigsqcup_{Q \in \CT^{\mathbf{0}}_{\beta} (\lambda/\mu,\Phi)} \mathcal{B}_{\tau} (\widehat{\beta(Q)}^{\dagger}.$$

    \item 
    We have the following disjoint union decomposition 
    $$ \RPP(\lambda/\mu, \Phi)=\displaystyle\bigsqcup_{ \alpha \in \mathbb{Z}^n _{+}} \RPP^{\alpha}(\lambda/\mu, \Phi) ,$$ where $\RPP^{\alpha}(\lambda/\mu,\Phi)$ be the set of all reverse plane partitions $R$ in $\RPP(\lambda/\mu,\Phi)$ satisfying $\ceq(R)=\alpha$. Combining the identification $\RPP^{\alpha}(\lambda/\mu, \Phi)=\SVRPP^{\alpha}_{\mathbf{0}}(\lambda/\mu, \Phi)$ with Theorem~\ref{theorem:main}, we recover the Demazure crystal structure on $\RPP^{\alpha}(\lambda/\mu,\Phi)$ \cite[Theorem 1]{Sidhu}
    $$ \RPP(\lambda/\mu, \Phi)  \cong \displaystyle\bigsqcup_{\alpha \in \mathbb{Z}^n _{+}} \displaystyle\bigsqcup_{Q \in \CT^{\alpha}_{\mathbf{0}} (\lambda/\mu,\Phi)} \mathcal{B}_{\tau} (\widehat{\beta(Q)}^{\dagger}.$$

    \item Applying Theorem~\ref{theorem:main} to the identity $\Tab(\lambda/\mu, \Phi)=\SVRPP^{\mathbf{0}}_{\mathbf{0}}(\lambda/\mu, \Phi)$, we obtain the following Demazure crystal decomposition $\Tab(\lambda/\mu, \Phi) $, originally noted in \cite[Theorem 3.11 \& Appendix]{KRSV}:
    $$ \Tab(\lambda/\mu, \Phi)  \cong \displaystyle\bigsqcup_{Q \in \CT^{\mathbf{0}}_{\mathbf{0}} (\lambda/\mu,\Phi)} \mathcal{B}_{\tau} (\widehat{\beta(Q)}^{\dagger}.$$
\end{enumerate}
\end{corollary}
\begin{corollary}
\label{corollary:main}
Let $\Phi \in \mathcal{F}[n]$ and $\lambda/\mu$ be a skew shape with $\lambda, \mu \in \mathcal{P}_n$. Then
$$ H_{\lambda/\mu}(\mathbf{x}_{\Phi};\mathbf{t};\mathbf{w})=\displaystyle\sum_{\alpha, \beta \in \mathbb{Z}^n _{+}}\mathbf{t}^{\alpha}\mathbf{w}^{\beta}\displaystyle\sum_{Q \in \CT^{\alpha}_{\beta} (\lambda/\mu,\Phi)}\kappa_{\widehat{\beta(Q)}}.$$
The above expression generalizes the following results:
\begin{itemize}
    \item Setting $\mathbf{t}=\mathbf{0}$ an replacing $\mathbf{w}$ by $-\mathbf{w}=(-w_1,\dots,-w_n)$ in $H_{\lambda/\mu}(\mathbf{x}_{\Phi};\mathbf{t};\mathbf{w})$, we recover the following expansion of the flagged refined skew stable Grothendieck polynomial $G_{\lambda/\mu}(X_{\Phi};\mathbf{w})$, given in \cite[\S4, Corollary 2]{Sidhu:SVT}
    $$G_{\lambda/\mu}(X_{\Phi};\mathbf{w})=H_{\lambda/\mu}(\mathbf{x}_{\Phi};\mathbf{0};-\mathbf{w})=\displaystyle\sum_{\beta \in \mathbb{Z}^n _{+}} (-1)^{|\beta|}\mathbf{w}^{\beta} \displaystyle\sum_{Q \in \CT^{\mathbf{0}}_{\beta} (\lambda/\mu,\Phi)}\kappa_{\widehat{\beta(Q)}}. $$

    \item Specializing at $\mathbf{w}=\mathbf{0}$ yields the expansion of the flagged refined dual stable Grothendieck polynomial $g_{\lambda/\mu}(X_{\Phi};\mathbf{t})$, established in \cite[\S5, Corollary 4]{Sidhu:SVT}, as follows:
    $$g_{\lambda/\mu}(X_{\Phi};\mathbf{t})=H_{\lambda/\mu}(\mathbf{x}_{\Phi};\mathbf{t};\mathbf{0})=\displaystyle\sum_{\alpha \in \mathbb{Z}^n _{+}} \mathbf{t}^{\alpha}\displaystyle\sum_{Q \in \CT^{\alpha}_{\mathbf{0}} (\lambda/\mu,\Phi)}\kappa_{\widehat{\beta(Q)}}.$$
    \item At $\mathbf{t}=\mathbf{0}$ and $\mathbf{w}=\mathbf{0}$, the polynomial $H_{\lambda/\mu}(\mathbf{x}_{\Phi};\mathbf{t};\mathbf{w})$ specializes to the flagged skew Schur polynomial $s_{\lambda/\mu}(X_{\Phi})$, yielding the key polynomial decomposition, given by Reiner and Shimozono \cite[Theorem 20]{RS}:
    $$ s_{\lambda/\mu}(X_{\Phi})=H_{\lambda/\mu}(\mathbf{x}_{\Phi};\mathbf{0};\mathbf{0})=\displaystyle\sum_{Q \in \CT^{\mathbf{0}}_{\mathbf{0}} (\lambda/\mu,\Phi)}\kappa_{\widehat{\beta(Q)}}.$$
\end{itemize}
\end{corollary}
\begin{example}
    Let $\lambda=(2,2), \mu=(1,0)$ and $\Phi=(1,2)$. Then the highest weight elements of $\SVRPP(\lambda/\mu, \Phi)$ (see \cite[Appendix~A]{Hybrid}) are the following:
$$ P =\ytableausetup{mathmode,
notabloids}
\begin{ytableau}
    \none &1\\ 1 &\textcolor{red}{1}   
\end{ytableau}
\hspace{0.5 cm}
Q =\ytableausetup{mathmode,
notabloids}
\begin{ytableau}
 \none &1\\1&2
\end{ytableau}
\hspace{0.5 cm}
R =\ytableausetup{mathmode,
notabloids}
\begin{ytableau}
     \none &1\\1 &\textcolor{red}{1},2  
\end{ytableau}
\hspace{0.5 cm}
S =\ytableausetup{mathmode,
notabloids}
\begin{ytableau}
    \none &1\\1,2 &2   
\end{ytableau}.
$$
Then we have
$$ 
\Tilde{P} =\ytableausetup{mathmode,
notabloids}
\begin{ytableau}
     \none &1\\ 1    
\end{ytableau}
\hspace{0.5 cm}
\Tilde{Q} =\ytableausetup{mathmode,
notabloids}
\begin{ytableau}
\none & \none &1\\1&2   
\end{ytableau}
\hspace{0.5 cm}
\Tilde{R} =\ytableausetup{mathmode,
notabloids}
\begin{ytableau}
    \none &1\\1  \\ 2  
\end{ytableau}
\hspace{0.5 cm}
\Tilde{S} =\ytableausetup{mathmode,
notabloids}
\begin{ytableau}
    \none & \none &1\\1 &2 \\ 2 
\end{ytableau}.
$$
It is clear that $\Tilde{P},\Tilde{Q}$ respects the flag $\Phi^{(0,0)}=(1,2)$ and  $\Tilde{R},\Tilde{S}$ respects the flag $\Phi^{(0,1)}=(1,2,2)$. Now
$$(
\begin{bmatrix}
    \mathbf{b}\Big([\lambda/\mu]^{(1,0)}_{(0,0)}\Big) \\
    \Read(\Tilde{P}) \\
\end{bmatrix} \rightarrow \emptyset)
=(\ytableausetup{mathmode,
notabloids}
  \begin{ytableau}
    1&1   
  \end{ytableau},\ytableausetup{mathmode,
notabloids}
  \begin{ytableau}
    1&2   
  \end{ytableau} ),
  \hspace{0.5 cm}
(
\begin{bmatrix}
    \mathbf{b}\Big([\lambda/\mu]^{(0,0)}_{(0,0)}\Big) \\
    \Read(\Tilde{Q}) \\
\end{bmatrix} \rightarrow \emptyset)
=(\ytableausetup{mathmode,
notabloids}
  \begin{ytableau}
    1&1\\2   
  \end{ytableau},\ytableausetup{mathmode,
notabloids}
  \begin{ytableau}
    1&2\\2
  \end{ytableau} ), 
$$
$$
(
\begin{bmatrix}
     \mathbf{b}\Big([\lambda/\mu]^{(1,0)}_{(0,1)}\Big) \\
    \Read(\Tilde{R}) \\       
\end{bmatrix} \rightarrow \emptyset)
=(\ytableausetup{mathmode,
notabloids}
  \begin{ytableau}
    1&1\\2   
  \end{ytableau},\ytableausetup{mathmode,
notabloids}
  \begin{ytableau}
    1&2\\3   
  \end{ytableau} ),
\hspace{0.5 cm}  
(
\begin{bmatrix}
    \mathbf{b}\Big([\lambda/\mu]^{(0,0)}_{(0,1)}\Big) \\
    \Read(\Tilde{S}) \\
\end{bmatrix} \rightarrow \emptyset)
=(\ytableausetup{mathmode,
notabloids}
  \begin{ytableau}
    1&1\\2&2  
  \end{ytableau},\ytableausetup{mathmode,
notabloids}
  \begin{ytableau}
    1&2\\2 &3
  \end{ytableau} ).
$$
Then the set of all $(\lambda/\mu,\Phi)$-compatible tableaux for SVRPP contains the following tableaux:
$$ \hat{P} =\ytableausetup{mathmode,
notabloids}
\begin{ytableau}
     1 &2  
\end{ytableau}
\text{ with } \overline{\ceq}(\hat{P})=(1,0), \overline{\ex}(\hat{P})=(0,0),$$
$$
\hat{Q} =\ytableausetup{mathmode,
notabloids}
\begin{ytableau}
    1&2 \\ 2   
\end{ytableau}
\text{ with }\overline{\ceq}(\hat{Q})=(0,0), \overline{\ex}(\hat{Q})=(0,0),
$$ 
$$
\hat{R} =\ytableausetup{mathmode,
notabloids}
\begin{ytableau}
    1&2 \\3   
\end{ytableau}
\text{ with } \overline{\ceq}(\hat{R})=(1,0), \overline{\ex}(\hat{R})=(0,1),$$
$$
\hat{S} =\ytableausetup{mathmode,
notabloids}
\begin{ytableau}
    1&2 \\2&3  
\end{ytableau}
\text{ with }\overline{\ceq}(\hat{S})=(0,0), \overline{\ex}(\hat{S})=(0,1).
$$
Then the corresponding left-key tableaux (using \cite{Willis:left-key}, \cite{Mrigendra:left-key}) are the following:
$$ K_{\_}(\hat{P}) =\ytableausetup{mathmode,
notabloids}
\begin{ytableau}
     1 &1   
\end{ytableau}
\text{ with } \beta(\hat{P})=(2,0),
\hspace{0.5 cm}
K_{\_}(\hat{Q}) =\ytableausetup{mathmode,
notabloids}
\begin{ytableau}
    1&2 \\ 2  
\end{ytableau}
\text{ with } \beta(\hat{Q})=(1,2),
$$ 
$$
K_{\_}(\hat{R}) =\ytableausetup{mathmode,
notabloids}
\begin{ytableau}
    1&1 \\3   
\end{ytableau}
\text{ with } \beta(\hat{R})=(2,0,1),
\hspace{0.5 cm}
K_{\_}(\hat{S}) =\ytableausetup{mathmode,
notabloids}
\begin{ytableau}
    1&1 \\2&2  
\end{ytableau}
\text{ with } \beta(\hat{S})=(2,2).
$$
Now using Theorem~\ref{theorem:RS-hat} we have the following:
$$\beta(\hat{P})=(2,0), \Phi^{(0,0)}=(1,2) \implies \widehat{\beta(\hat{P})}=(2,0);$$
$$\beta(\hat{Q})=(1,2), \Phi^{(0,0)}=(1,2) \implies \widehat{\beta(\hat{Q})}=(1,2);$$ 
$$\beta(\hat{R})=(2,0,1), \Phi^{(0,1)}=(1,2,2) \implies \widehat{\beta(\hat{R})}=(2,1);$$
$$\beta(\hat{S})=(2,2,0), \Phi^{(0,1)} =(1,2,2) \implies\widehat{\beta(\hat{S})}=(2,2).$$
Therefore, we have the following expansion
$$  H_{\lambda/\mu}(\mathbf{x}_{\Phi};\mathbf{t};\mathbf{w})=t_1\kappa_{(2,0)} + \kappa_{(1,2)} +t_1w_2 \kappa_{(2,1)} + w_2 \kappa_{(2,2)}.$$
\end{example}
\begin{corollary}
\label{corollary:schur-SVT}
Let $\lambda/\mu$ $(\lambda, \mu \in \mathcal{P}_n)$ be a skew shape and $\Phi= (\underbrace{n,n,\dots,n}_{n})$. Then 
$$ H_{\lambda/\mu}(\characx;\mathbf{t};\mathbf{w})=\displaystyle\sum_{ \alpha ,\beta \in \mathbb{Z}^n _{+}}\mathbf{t}^{\alpha}\mathbf{w}^{\beta}\displaystyle\sum_{Q \in \CT^{\alpha}_{\beta} (\lambda/\mu,\Phi)}s_{\widehat{\beta(Q)}^{\dagger}}$$
$$ \hspace{2.7 cm} =\displaystyle\sum_{\alpha , \beta \in \mathbb{Z}^n _{+}}\mathbf{t}^{\alpha}\mathbf{w}^{\beta}\displaystyle\sum_{Q \in \CT^{\alpha}_{\beta} (\lambda/\mu,\Phi)}s_{\sh(Q)}.$$
\end{corollary}
\section{Applications}
\label{Section 4}
In this section, we expand the hybrid Grothendieck polynomial $H_{\lambda / \mu}(\characx;\mathbf{t};\mathbf{w})$ into two distinct bases using methods from \cite{Morse-Jason:K-bases}:
\begin{enumerate}
    \item the stable Grothendieck polynomials $G_{\nu}(\mathbf{x})$,
    \item the dual stable Grothendieck polynomials $g_{\nu}(\characx)$.
\end{enumerate}
Analogous expansions have been obtained for the canonical Grothendieck polynomials in \cite{Anne:un-hook}, as well as for the row-refined skew stable Grothendieck polynomials, the refined dual stable Grothendieck polynomials, and the Schur P-functions in \cite[\S5]{Sidhu:SVT}.
\begin{definition}
    \begin{upshape}
        \cite{Morse-Jason:K-bases}
    \end{upshape}
A symmetric function, $f_{\gamma}$ has a \emph{tableaux Schur expansion} if there exist a set of semi-standard Young tableaux $\mathbb{T}_\gamma$ and a weight function $\wt_{\gamma}$ such that 
\begin{equation}
    \label{eq:X}
    f_{\gamma}=\displaystyle\sum_{T \in \mathbb{T}(\gamma)} \wt_{\gamma}(T)s_{\sh(T)}(\mathbf{x}).
\end{equation}
\end{definition}
Given $\mathbb{T}(\gamma)$, we define the corresponding sets of semi-standard set-valued tableaux $\mathbb{S}(\gamma)$ and reverse plane partitions $\mathbb{R}(\gamma)$ (of partition shapes) by requiring that their reading words rectify to elements of $\mathbb{T}(\gamma)$, that is,
$$ \mathbb{S}(\gamma)=\{ S: \rect(\Read(S)) \in \mathbb{T}(\gamma) \} ,$$
$$ \mathbb{R}(\gamma)=\{ R: \rect(\Read(R)) \in \mathbb{T}(\gamma) \} .$$
Similarly, we extend $\wt_{\gamma}$ to both $\mathbb{S}(\gamma)$ and $\mathbb{R}(\gamma)$ by setting $\wt_{\gamma}(X):=\wt_{\gamma}(\rect(\Read(X)))$.
\begin{theorem}\begin{upshape}
    \cite[Theorem 3.5]{Morse-Jason:K-bases}
\end{upshape}
\label{theorem:K bases}
Assuming the expansion of $f_{\gamma}$ in \eqref{eq:X}, 
we have $$f_{\gamma} = \displaystyle\sum_{R \in \mathbb{R}(\gamma)} \wt_{\gamma}(R)G_{\sh(R)}(\mathbf{x}),$$
$$f_{\gamma} = \displaystyle\sum_{S \in \mathbb{S}(\gamma)} (-1)^{|\ex(S)|}\wt_{\gamma}(S) g_{\sh(S)}(\mathbf{x}).$$
\end{theorem}
For $\alpha , \beta \in \mathbb{Z}^n_{+}$, we define $\mathbb{T}^{\alpha}_{\beta}(\lambda/\mu):=\CT^{\alpha}_{\beta}(\lambda/\mu, \Phi)$, where $\Phi= (\underbrace{n,n,\dots,n}_{n})$
and for all $P \in \mathbb{T}^{\alpha}_{\beta}(\lambda/\mu)$, define $\wt^{\alpha,\beta}_{\lambda/\mu}(P):=\mathbf{t}^{\alpha}\mathbf{w}^{\beta}$. Then Corollary~\ref{corollary:schur-SVT} gives the following tableau Schur expansion of $H_{\lambda/\mu}(\characx;\mathbf{t};\mathbf{w})$
\begin{equation}
\label{set value:eq}
 H_{\lambda/\mu}(\characx;\mathbf{t};\mathbf{w})=\displaystyle\sum_{\alpha , \beta \in \mathbb{Z}^{n} _{+}}\displaystyle\sum_{P \in \mathbb{T}^{\alpha}_{\beta}(\lambda/\mu)} \wt^{\alpha,\beta}_{\lambda/\mu}(P) s_{\sh(P)}.
\end{equation}
So by Theorem~\ref{theorem:K bases}, we have
$$H_{\lambda/\mu}(\characx;\mathbf{t};\mathbf{w})= \displaystyle\sum_{{\alpha , \beta \in \mathbb{Z}^{n} _{+}}} \displaystyle \sum_{R \in \mathbb{R}^{\alpha}_{\beta}(\lambda/\mu)} \wt^{\alpha,\beta}_{\lambda/\mu}(R)G_{\sh(R)}(\characx)$$
$$\hspace{2.9 cm}=\displaystyle\sum_{{\alpha , \beta \in \mathbb{Z}^{n} _{+}}}\displaystyle \sum_{S \in \mathbb{S}^{\alpha}_{\beta}(\lambda/\mu)} (-1)^{|\ex(S)|} \wt^{\alpha,\beta}_{\lambda/\mu}(S)g_{\sh(S)}(\characx),$$
where $\mathbb{R}^{\alpha}_{\beta}(\lambda/\mu)$ and $\mathbb{S}^{\alpha}_{\beta}(\lambda/\mu)$ denote the sets of reverse plane partitions and semi-standard set-valued tableaux, respectively, whose reading words are Knuth-equivalent to an element of $\mathbb{T}^{\alpha}_{\beta}(\lambda/\mu)$.
\begin{example}
Let $\lambda=(2,2), \mu=(1)$. Then the highest weight elements of $\SVRPP_2(\lambda/\mu)$ (see \cite[Appendix~A]{Hybrid}) are the following:
$$
P_1=\ytableausetup{mathmode,
notabloids}
\begin{ytableau}
    \none &1 \\ 1 &\textcolor{red}{1}   
\end{ytableau}
  \hspace{0.5 cm}
P_2=\ytableausetup{mathmode,
notabloids}
\begin{ytableau}
    \none &1\\1 &2   
\end{ytableau}
  \hspace{0.5 cm}
P_3=\ytableausetup{mathmode,
notabloids}
\begin{ytableau}
    \none &1\\1 &\textcolor{red}{1},2   
\end{ytableau}
  \hspace{0.5 cm}
P_4=\ytableausetup{mathmode,
notabloids}
\begin{ytableau}
    \none &1\\1,2 &2   
\end{ytableau}
  \hspace{0.5 cm}
P_5=\ytableausetup{mathmode,
notabloids}
\begin{ytableau}
    \none &1,2\\1 &\textcolor{red}{2}   
\end{ytableau}
  \hspace{0.5 cm}
P_6=\ytableausetup{mathmode,
notabloids}
\begin{ytableau}
    \none &1,2\\1,2 &\textcolor{red}{2}   
\end{ytableau}.
$$
Then we obtain
$$ \Tilde{P_1}=\ytableausetup{mathmode,
notabloids}
\begin{ytableau}
    \none &1 \\ 1   
\end{ytableau}
  \hspace{0.3 cm}
\Tilde{P_2} =\ytableausetup{mathmode,
notabloids}
\begin{ytableau}
    \none & \none &1\\1 &2   
\end{ytableau}
  \hspace{0.3 cm}
\Tilde{P_3} =\ytableausetup{mathmode,
notabloids}
\begin{ytableau}
    \none &1\\1 \\ 2   
\end{ytableau}
\hspace{0.3 cm}
\Tilde{P_4} =\ytableausetup{mathmode,
notabloids}
\begin{ytableau}
   \none & \none &1\\1 &2 \\ 2   
\end{ytableau}
\hspace{0.3 cm}
\Tilde{P_5} =\ytableausetup{mathmode,
notabloids}
\begin{ytableau}
  \none &1 \\  \none &2 \\1
\end{ytableau}
\hspace{0.3 cm}
\Tilde{P_6} =\ytableausetup{mathmode,
notabloids}
\begin{ytableau}
  \none  & 1\\ \none &2\\1  \\ 2   
\end{ytableau}.
$$
Now the column insertion (Theorem~\ref{Theorem:Burge}) gives
$$(
\begin{bmatrix}
    \mathbf{b}\Big([\lambda/\mu]^{(1,0)}_{(0,0)}\Big) \\
    \Read(\Tilde{P_1}) \\
\end{bmatrix}
\rightarrow \emptyset)
=(\ytableausetup{mathmode,
notabloids}
  \begin{ytableau}
    1&1   
  \end{ytableau},\ytableausetup{mathmode,
notabloids}
  \begin{ytableau}
    1&2  
  \end{ytableau} ), 
\hspace{0.5 cm}
(
\begin{bmatrix}
    \mathbf{b}\Big([\lambda/\mu]^{(0,0)}_{(0,0)}\Big) \\
    \Read(\Tilde{P_2}) \\
\end{bmatrix}
\rightarrow \emptyset)
=(\ytableausetup{mathmode,
notabloids}
  \begin{ytableau}
    1&1\\2   
  \end{ytableau},\ytableausetup{mathmode,
notabloids}
  \begin{ytableau}
    1&2\\2   
  \end{ytableau} ),$$
$$(
\begin{bmatrix}
    \mathbf{b}\Big([\lambda/\mu]^{(1,0)}_{(0,1)}\Big) \\
    \Read(\Tilde{P_3}) \\
\end{bmatrix}
\rightarrow \emptyset)
=(\ytableausetup{mathmode,
notabloids}
  \begin{ytableau}
    1&1\\2   
  \end{ytableau},\ytableausetup{mathmode,
notabloids}
  \begin{ytableau}
    1&2\\3   
  \end{ytableau} ), 
\hspace{0.5 cm}
(
\begin{bmatrix}
    \mathbf{b}\Big([\lambda/\mu]^{(0,0)}_{(0,1)}\Big) \\
    \Read(\Tilde{P_4}) \\
\end{bmatrix}
\rightarrow \emptyset)
=(\ytableausetup{mathmode,
notabloids}
  \begin{ytableau}
    1&1\\2&2   
  \end{ytableau},\ytableausetup{mathmode,
notabloids}
  \begin{ytableau}
    1&2\\2&3   
  \end{ytableau} ),$$
$$(
\begin{bmatrix}
    \mathbf{b}\Big([\lambda/\mu]^{(1,0)}_{(1,0)}\Big) \\
    \Read(\Tilde{P_5}) \\
\end{bmatrix}
\rightarrow \emptyset)
=(\ytableausetup{mathmode,
notabloids}
  \begin{ytableau}
    1&1\\2   
  \end{ytableau},\ytableausetup{mathmode,
notabloids}
  \begin{ytableau}
    1&3\\2   
  \end{ytableau} ), 
\hspace{0.5 cm}
(
\begin{bmatrix}
    \mathbf{b}\Big([\lambda/\mu]^{(1,0)}_{(1,1)}\Big) \\
    \Read(\Tilde{P_6}) \\
\end{bmatrix}
\rightarrow \emptyset)
=(\ytableausetup{mathmode,
notabloids}
  \begin{ytableau}
    1&1\\2 &2  
  \end{ytableau},\ytableausetup{mathmode,
notabloids}
\begin{ytableau}
    1&3\\2 &4  
\end{ytableau} ).$$ 
Therefore, we get
$$ 
\mathbb{T}^{(1,0)}_{(0,0)}(\lambda/\mu)=\{
\ytableausetup{mathmode,
notabloids}
  \begin{ytableau}
    1&2   
  \end{ytableau}
\}, \quad 
 \mathbb{T}^{(0,0)}_{(0,0)}(\lambda/\mu)=\{
\ytableausetup{mathmode,
notabloids}
  \begin{ytableau}
    1&2 \\2  
  \end{ytableau}
\}, \quad 
\mathbb{T}^{(1,0)}_{(0,1)}(\lambda/\mu)=\{
\ytableausetup{mathmode,
notabloids}
  \begin{ytableau}
    1&2 \\3  
  \end{ytableau}
\},
$$
$$
\mathbb{T}^{(0,0)}_{(0,1)}(\lambda/\mu)=\{
\ytableausetup{mathmode,
notabloids}
  \begin{ytableau}
    1&2 \\2&3  
  \end{ytableau}
\}, \quad 
\mathbb{T}^{(1,0)}_{(1,0)}(\lambda/\mu)=\{
\ytableausetup{mathmode,
notabloids}
  \begin{ytableau}
    1&3\\2   
  \end{ytableau}
\}, \quad 
\mathbb{T}^{(1,0)}_{(1,1)}(\lambda/\mu)=\{
\ytableausetup{mathmode,
notabloids}
  \begin{ytableau}
    1&3\\2&4   
  \end{ytableau}
\}.
$$
Thus we have
$$ \mathbb{R}^{(1,0)}_{(0,0)}(\lambda/\mu) =
\{
\begin{ytableau}
   1&2   
\end{ytableau},
\begin{ytableau}
   1&2\\ \textcolor{red}{1}   
\end{ytableau},
\begin{ytableau}
   1&2\\ \textcolor{red}{1}& \textcolor{red}{2}   
\end{ytableau},\dots
 \}, \quad
\mathbb{R}^{(0,0)}_{(0,0)}(\lambda/\mu)
=\{
\begin{ytableau}
   1&2\\2   
  \end{ytableau},
\begin{ytableau}
   1&2 \\ \textcolor{red}{1} \\2   
  \end{ytableau},
\begin{ytableau}
   1&2\\2\\\textcolor{red}{2}  
\end{ytableau},\dots  
\},$$
$$ \mathbb{R}^{(1,0)}_{(0,1)}(\lambda/\mu) =
\{
\begin{ytableau}
   1&2\\3   
\end{ytableau},
\begin{ytableau}
   1&2\\ \textcolor{red}{1} \\3   
\end{ytableau},
\begin{ytableau}
   1&2\\3\\ \textcolor{red}{3}   
\end{ytableau},\dots
 \}, \quad
\mathbb{R}^{(0,0)}_{(0,1)}(\lambda/\mu)
=\{
\begin{ytableau}
   1&2\\2&3   
\end{ytableau},
\begin{ytableau}
   1&2\\ \textcolor{red}{1} &3\\2   
\end{ytableau},
\begin{ytableau}
   1&2\\2&3\\ \textcolor{red}{2}  
\end{ytableau},\dots  
\},$$
$$
\mathbb{R}^{(1,0)}_{(1,0)}(\lambda/\mu) =
\{
\begin{ytableau}
   1&3\\2   
\end{ytableau},
\begin{ytableau}
   1&3\\ \textcolor{red}{1} \\2   
\end{ytableau},
\begin{ytableau}
   1&3\\2\\ \textcolor{red}{2}   
\end{ytableau},\dots
 \}, \quad
\mathbb{R}^{(1,0)}_{(1,1)}(\lambda/\mu)
=\{
\begin{ytableau}
   1&3\\2&4   
  \end{ytableau},
\begin{ytableau}
   1&3\\ \textcolor{red}{1} &4\\2   
  \end{ytableau},
\begin{ytableau}
   1&3\\2&4\\ \textcolor{red}{2}   
\end{ytableau},\dots  
\}.$$
Therefore we obtain

\noindent$H_{\lambda/\mu}(\characx;\mathbf{t};\mathbf{w})=t_1\Big(G_{(2)}(\characx) + G_{(2,1)}(\characx)+G_{(2,2)}(\characx)+\cdots\Big)+ \Big(G_{(2,1)}(\characx) + 2G_{(2,1,1)}(\characx) + \cdots \Big) + t_1w_2\Big(G_{(2,1)}(\characx) +2 G_{(2,1,1)}(\characx) + \cdots\Big)+ w_2\Big(G_{(2,2)}(\characx)+2G_{(2,2,1)} +\cdots \Big) +
t_1w_1 \Big(G_{(2,1)}(\characx)+2G_{(2,1,1)} +\cdots \Big)+ t_1w_1w_2\Big(G_{(2,2)}(\characx)+2G_{(2,2,1)} +\cdots\Big)
$.

\noindent Also, we have
$$
\mathbb{S}^{(1,0)}_{(0,0)}(\lambda/\mu) =
\{
\begin{ytableau}
    1&2   
\end{ytableau}
\}, \quad 
\mathbb{S}^{(0,0)}_{(0,0)}(\lambda/\mu) =
\{\begin{ytableau}
    1&2 \\2  
\end{ytableau},
\begin{ytableau}
     1,2 &2   
\end{ytableau}
\},  
$$
$$
\mathbb{S}^{(1,0)}_{(0,1)}(\lambda/\mu) =
\{\begin{ytableau}
    1&2\\3   
\end{ytableau},
\begin{ytableau}
     1 &2,3   
\end{ytableau}
  \}, \quad
\mathbb{S}^{(0,0)}_{(0,1)}(\lambda/\mu) =
\{\begin{ytableau}
    1&2 \\2&3  
\end{ytableau},
\begin{ytableau}
     1&2,3 \\2   
\end{ytableau}
 \}, 
$$
$$
\mathbb{S}^{(1,0)}_{(1,0)}(\lambda/\mu) =
\{\begin{ytableau}
    1&3\\2   
\end{ytableau},
\begin{ytableau}
     1,2 &3   
\end{ytableau}
\}, \quad 
\mathbb{S}^{(1,0)}_{(1,1)}(\lambda/\mu) =
\{
\begin{ytableau}
    1&3\\2&4   
\end{ytableau},
\begin{ytableau}
     1 &3,4\\2   
\end{ytableau}
\}.  
$$
Hence we get

\noindent$H_{\lambda/\mu}(\characx;\mathbf{t};\mathbf{w})=t_1
g_{(2)}(\characx)+\Big(g_{(2,1)}(\characx)-g_{(2)}(\characx) \Big)+ t_1w_2\Big( g_{(2,1)}(\characx) -g_{(2)}(\characx) \Big) + w_2\Big( g_{(2,2)}(\characx) -g_{(2,1)}(\characx) \Big) +t_1w_1\Big( g_{(2,1)}(\characx) -g_{(2)}(\characx) \Big) +t_1w_1w_2\Big( g_{(2,2)}(\characx) -g_{(2,1)}(\characx) \Big) $.
\end{example}

\end{document}